\documentclass[11pt]{amsart}

\usepackage{amssymb}
\usepackage{amsthm}
\usepackage{amsmath}
\usepackage{graphicx}
\usepackage{amsaddr}
\usepackage{comment}
\usepackage[usenames,dvipsnames]{xcolor}
\usepackage[margin = 1in]{geometry}
\usepackage{enumerate}
\usepackage[toc,page]{appendix}
\usepackage{lineno}
\usepackage{color}
\usepackage{bbm}
\usepackage{hyperref}
\usepackage{mhchem}
\usepackage{siunitx}

\newcommand{\blue}{}
\definecolor{mygreen}{rgb}{0.1,0.75,0.2}

\providecommand{\bbs}[1]{\left(#1\right)}

 \newtheorem{thm}{Theorem}[section]
 \newtheorem{cor}[thm]{Corollary}
 \newtheorem{lem}[thm]{Lemma}
 \newtheorem{prop}[thm]{Proposition}

 \theoremstyle{definition}
 \newtheorem{defn}[thm]{Definition}

 \theoremstyle{remark}
 \newtheorem{rem}[thm]{Remark}

 \numberwithin{equation}{section}

\DeclareMathOperator{\KL}{KL}
\DeclareMathOperator{\ran}{Ran}
\DeclareMathOperator{\kk}{Ker}

\DeclareMathOperator{\argmax}{argmax}

\newcommand{\peq}{\vec{x}^{\text{s}}}

\newcommand{\la}{\langle}
\newcommand{\ra}{\rangle}
\newcommand{\pt}{\partial}

\newcommand{\eps}{\varepsilon}

\newcommand{\ud}{\,\mathrm{d}}
\newcommand{\8}{\infty}

\newcommand{\bR}{\mathbb{R}}
\newcommand{\bZ}{\mathbb{Z}}
\newcommand{\bE}{\mathbb{E}}
\newcommand{\bP}{\mathbb{P}}
\newcommand{\vs}{\vec{s}}
\newcommand{\vx}{\vec{x}}
\newcommand{\psih}{u_{\scriptscriptstyle{\text{h}}}}
\newcommand{\ph}{\psi_{\scriptscriptstyle{\text{h}}}}
\newcommand{\uh}{u_{\scriptscriptstyle{\text{h}}}}
\newcommand{\vh}{v_{\scriptscriptstyle{\text{h}}}}
\newcommand{\psihs}{\psi_{\scriptscriptstyle{\text{h}}}^{\scriptscriptstyle{\text{ss}}}}
\newcommand{\V}{\scriptscriptstyle{\text{h}}}

\newcommand{\cv}{X^{\scriptscriptstyle{\text{h}}}}

\newcommand{\vxv}{\vec{x}_{i}}
\newcommand{\vxi}{\vec{x}_i}
\newcommand{\vyi}{\vec{y}_i}

\newcommand{\vp}{\vec{p}}
\newcommand{\vy}{\vec{y}}

\newcommand{\vm}{\vec{m}}

\newcommand{\vq}{\vec{q}}

\newcommand{\usc}{\text{USC}(\bR^N)}
\newcommand{\lsc}{\text{LSC}(\bR^N)}
\newcommand{\buc}{C(\bR^{N*})}
\newcommand{\ccc}{C_{c}(\mathbb{R}^{N*})}
\newcommand{\ellb}{\ell^\8(\Omega_{\V}^*)}
\newcommand{\ellc}{\ell^\8_{c}(\Omega_{\V}^*)}

\newcommand{\omp}{\Omega_{\V}^+}
\newcommand{\omn}{\Omega_{\V}\backslash \Omega_{\V}^+}

\begin{document}

\title[Large deviation for chemical reactions]{Large deviation principle and thermodynamic limit of chemical master equation via nonlinear semigroup}

\author[Y. Gao]{Yuan Gao}
\address{Department of Mathematics, Purdue University, West Lafayette, IN}
\email{gao662@purdue.edu}

\author[J.-G. Liu]{Jian-Guo Liu}
\address{Department of Mathematics and Department of
  Physics, Duke University, Durham, NC}
\email{jliu@math.duke.edu}




\date{\today}

\begin{abstract}
Chemical reactions can be modeled by a random time-changed Poisson process on countable states. The macroscopic behaviors, such as large fluctuations, can be studied via the WKB reformulation. The WKB reformulation for the backward equation is Varadhan's discrete nonlinear semigroup and is also a monotone scheme that approximates the limiting first-order Hamilton-Jacobi equations (HJE).
The discrete Hamiltonian is an m-accretive operator, which generates a nonlinear semigroup on countable grids and justifies the well-posedness of the chemical master equation (CME) and the backward equation with 'no reaction' boundary conditions.
The convergence from the monotone schemes to the viscosity solution of HJE is proved by constructing barriers to overcome the polynomial growth coefficients in the Hamiltonian. This implies the convergence of Varadhan's discrete nonlinear semigroup to the continuous Lax-Oleinik semigroup and leads to the large deviation principle for the chemical reaction process at any single time.
 Consequently, the macroscopic mean-field limit reaction rate equation is recovered with a concentration rate estimate. Furthermore, we establish the convergence from a reversible invariant measure to an upper semicontinuous viscosity solution of the stationary HJE.
\end{abstract}

\maketitle 

\section{Introduction}

Chemical or biochemical reactions, such as the production of useful materials in industry and the maintenance of metabolic processes with enzymes in living cells, are among the most important events in the world. At a microscopic scale, these reactions can be understood from a probabilistic viewpoint. 
{\blue In this paper, we focus on the large deviation principle (fluctuation estimate) for   chemical reaction processes in the thermodynamic limit regime. As a byproduct, the reaction rate equation will be recovered as a mean-field limit equation with an exponential concentration rate. To estimate the small fluctuations away from the typical chemical reaction trajectory, we will utilize Varadhan's exponential nonlinear semigroup method with novel techniques, which we will explain in detail below.}

A convenient way to stochastically describe chemical reactions is via random time-changed Poisson processes $X(t)$ \eqref{Csde}, c.f. \cite{Kurtz15}. This continuous time Markov process on countable states counts the number of molecular species $X_i(t), i=1, \cdots, N$ for chemical reactions happening in a large container characterized by a size $\frac{1}{h}\gg 1$. We assume chemical reactions in a container are independent of molecule position, and the molecular number is proportional to the container size. Therefore, we will refer to the large size limit $h\to 0$ as the thermodynamic limit or macroscopic limit. After counting the net change of the molecular numbers $\vec{\nu}_j$ for a $j$-th reaction, the rescaled process $\cv(t)=hX(t)$ is
\begin{equation} \label{Csde}
\begin{aligned}
\cv(t) = \cv(0) + \sum_{j=1}^M  \vec{\nu}_{j} h \Bigg(\mathbbm{1}_{\{\cv(t^-)+\vec{\nu}_j h \geq 0\}} Y^+_j  \bbs{\frac{1}{h}\int_0^t \tilde{\Phi}^+_j(\cv(s))\ud s}\\
-\mathbbm{1}_{\{\cv(t^-)-\vec{\nu}_j h \geq 0\}}Y^-_j  \bbs{\frac{1}{h}\int_0^t \tilde{\Phi}^-_j(\cv(s))\ud s}\Bigg),
\end{aligned}
\end{equation}
where $Y^\pm_{j}(t)$ are i.i.d. unit rate Poisson processes, and $\mathbbm{1}$ is the indicating function to show that there is no reaction if the next state $\cv(t^-)\pm\vec{\nu}_j h$ is negative. This `no reaction' constraint will also be reflected in the chemical master equation \eqref{odex} below. The intensity $\tilde{\Phi}_j^{\pm}(\vx)$ of this Poisson process is given by the law of mass action (LMA) \eqref{newR}, indicating the encounter of species in one reaction. The chemical master equation (CME) \eqref{rp_eq} for the rescaled process $\cv(t)$ is a linear ordinary differential equation (ODE) system on countable discrete grids with a `no reaction' boundary condition to maintain nonnegative counting states and the conservation of total probability. The key observation is that CME \eqref{rp_eq} has a monotonicity property, which is also known as a monotone scheme approximation for hyperbolic differential equations. The generator and the associated backward equation \eqref{backward} of the process $\cv$ can be regarded as a dual equation for CME. The backward equation \eqref{backward} is also a linear ODE system on the same countable discrete grids, adapting the `no reaction' boundary condition and monotonicity.

The macroscopic behaviors of chemical reactions with the above stochastic modeling are deterministic "statistical properties" that can be studied by taking the large size limit $h\to 0$. At the first level, the law of large numbers characterizes the mean-field limit nonlinear ODE, known as the reaction rate equation (RRE), with polynomial nonlinearity due to the law of mass action (LMA):
  \begin{equation}\label{odex}
\frac{\ud}{\ud t} \vec{x} =  \sum_{j=1}^M \vec{\nu}_j \bbs{\Phi^+_j(\vec{x}) - \Phi^-_j(\vec{x})}, \qquad \Phi^\pm_j(\vxv)= k^\pm_j  \prod_{\ell=1}^{N} \bbs{x_\ell}^{\nu^\pm_{j\ell}}.
\end{equation}
At a more detailed level, the large deviation principle estimates the fluctuations away from the mean-field limit RRE \eqref{odex}. To capture the small probabilities in the large deviation regime, the WKB reformulation $p_{\V}(\vxv,t) =  e^{-\frac{\psi_{\V}(\vxv,t)}{h}}$ of the master equation \cite{Kubo73} gives a discrete Hamiltonian and an exponentially nonlinear ODE system on the same countable discrete grids; see \eqref{upwind00}. This nonlinear ODE system inherits the monotonicity and the "no reaction" boundary condition from the CME, so we refer to \eqref{upwind00} as the CME in the Hamilton-Jacobi equation (HJE) form or the monotone scheme for HJE \eqref{HJEpsi}. The backward equation can proceed with the same WKB reformulation and inherit the "no reaction" boundary condition as the restriction $\vxi\pm\vec{\nu_{j}} h\geq 0$. The resulting exponentially nonlinear ODE system is also a nonlinear semigroup \eqref{wkb_b} on discrete countable grids:
\begin{equation}\label{upwind}
\begin{aligned}
\pt_t \uh(\vxi,t)= & H_{\V}(\uh(\vxi),\vxi)\\
 := & \sum_{j=1, \vxi+\vec{\nu_{j}} h\geq 0}^M     \tilde{\Phi}^+_j(\vxv)\bbs{ e^{\frac{\uh(\vxv+ \vec{\nu_{j}} h)-\uh(\vxv)}{h}} - 1}    + \sum_{j=1, \vxi-\vec{\nu_{j}} h\geq 0}^M \tilde{\Phi}_j^-(\vxv)\bbs{ e^{  \frac{\uh(\vxv- \vec{\nu_{j}}h)-\uh(\vxv)}{h} } - 1}.
\end{aligned}
\end{equation}
 A zero extension to negative grids, which is consistent with the "no reaction" boundary condition, is taken on $\tilde{\Phi}^\pm_j$; see \eqref{exPHI} and Lemma \ref{lem:generator}. Thus, after proper extension, the domain for \eqref{upwind} and the corresponding  HJE \eqref{HJE2psi} becomes the whole space. We observe that \eqref{upwind} is a monotone scheme because $H_{\V}$ is decreasing w.r.t. $\uh(\vxi)$ while increasing w.r.t. $\uh(\vxi\pm \vec{\nu}_j h)$.  The probability representation for this nonlinear semigroup is given by 
 \begin{equation}\label{semigroup}
 \bbs{S_t u_0}(\vxv)=h \log \bE^{\vxv}\bbs{e^{ \frac{u_0(\cv_t)}{h}}},
 \end{equation}
dated back to \textsc{Varadhan} \cite{Varadhan_1966}. The justification of the existence of the nonlinear semigroup on discrete countable grids relies on the resolvent approximation $u_{\V} - \Delta t H_{\V}(u_{\V})=f_{\V}$ (i.e.,  backward Euler scheme \eqref{bEuler}) and \textsc{Crandall-Liggett}'s nonlinear semigroup theory \cite{CrandallLiggett}; see Theorem \ref{thm:backwardEC}.

{\blue We will prove the large deviation principle for process $\cv_t$ at a single time through Varadhan's inverse lemma \cite{Bryc_1990}. In detail, the rigorous justification for the WKB expansion of the backward equation can be obtained by proving the convergence from the discrete Varadhan's nonlinear semigroup of \eqref{upwind} to the Lax-Oleinik semigroup solution of the limiting HJE:
\begin{equation}\label{HJE2psi}
\pt_t u(\vx,t) - H(\nabla u(\vx),\vx)=0, \quad  H(\vp,\vx):=\sum_{j=1}^M \bbs{ \Phi^+_j(\vec{x})\bbs{e^{\vec{\nu}_j\cdot \vp}-1 }     +   \Phi_j^-(\vec{x})\bbs{ e^{-\vec{\nu}_j\cdot \vp}-1 }  }.
\end{equation}
Notice the Lax-Oleinik semigroup representation for \eqref{HJE2psi}
\begin{equation}\label{LO} 
u(\vx,t) = \sup_{\vy  } \bbs{u_0(\vy) - I(\vy; \vx,t)},  \quad I(\vy; \vx,t) := \inf_{\gamma(0)=\vx, \gamma(t)=\vy} \int_0^t L(\dot{\gamma}(s),\gamma(s)) \ud s,
\end{equation}
where $L(\vs,\vx)$ is the convex conjugate of Hamiltonian $H(\vp,\vx)$.
This connects the viscosity solution to the HJE with the pathwise deterministic optimal control problem with terminal profit $u_0(y)$ and running cost $L(\dot{\gamma}(s),\gamma(s))$; cf. \cite{evans2008weak}. Then, the large deviation principle can be obtained via the inverse Varadhan's lemma (proved by \textsc{Bryc} \cite{Bryc_1990}). In other words, the sufficient conditions for the large deviation principle are to justify the convergence of the above nonlinear semigroup and the exponential tightness of the process. These are also necessary conditions, known as  Varadhan's lemma \cite{Varadhan_1966}.

To obtain the above convergence, the viscosity method is a well-developed tool pioneered by \textsc{Crandall} and \textsc{Lions} in \cite{crandall1983viscosity, crandall1984two} using two methods: the vanishing viscosity method and the monotone scheme approximation. Our first contribution is to observe that the WKB reformulation for the backward equation is exactly a monotone scheme for the limiting first-order HJE \eqref{HJE2psi}. Then, one can prove the convergence from the monotone scheme \eqref{upwind} to the viscosity solution of  HJE \eqref{HJE2psi} using the upper/lower semicontinuous envelopes of the discrete solution to the resolvent problem and \textsc{Crandall-Liggett}'s nonlinear semigroup theory.
However, two difficulties arise: the "no-reaction" boundary condition and the coefficients $\Phi_j^\pm(\vx)$ in Hamiltonian $H$ have polynomial growth at far fields. The abstract theorems established by \textsc{Souganidis} \cite{souganidis1985approximation} and \textsc{Barles, Perthame} \cite{Barles_Perthame_1987} cannot be directly applied.

Second, for compactness of $u_{\V}$ at the far field, we need a one-point compactification at the far field. Hence, we choose the ambient space as $\buc$, i.e., the far field limit is a constant value (see \eqref{space}).
The key observation is that each chemical reaction satisfies the conservation of mass, i.e., there exists a positive mass vector $\vm$ such that $\vec{\nu}_j \cdot \vm=0, \, j=1,\cdots,M$. Thanks to this, any mass function $f(\vm\cdot \vx)$ is a stationary solution to the Hamilton-Jacobi equations (HJEs) \eqref{upwind} and \eqref{HJE2psi}, and hence can be used to construct barriers to overcome the polynomial growth of $\Phi^\pm_j(\vx)$.    Thus we first prove the viscosity solution $u(\vx,t)\in \ccc$ and then extend it to $\buc$ by using the non-expansive property of the resolvent problem; see Corollary \ref{cor:semi}.

Third, to overcome the dynamic boundary condition due to the "no reaction" constraint for negative states, we perform a zero extension for ${\Phi}^\pm_j(\vx)$ for $\vx\in \mathbb{R}^N$ and impose a local Lipschitz continuous condition. This condition excludes some reactions, but it is convenient to use to guarantee the comparison principle for the HJE. Removing this technical condition is possible by considering an optimal control formulation with a boundary cost, but we leave it for future study.

 Our convergence result (see Theorem \ref{thm_vis}) from the discrete nonlinear semigroup solution of \eqref{upwind} to the viscosity solution of HJE \eqref{HJE2psi} provides the convergence part of the large deviation principle at a single time. The exponential tightness is trivial if the starting point of $\cv_t$ is deterministic and positive due to the mass conservation law for the chemical reaction. In general, to verify the exponential tightness of $\cv$ at any time $t$, we impose one of the following two assumptions: one is the existence of a positive reversible invariant measure $\pi_{\V}$ (see \eqref{master_db_n}) that is exponentially tight; the other is that there is compact support for the initial density $\rho^0_{\V}$; see Theorem \ref{thm_dp}. As a consequence of the large deviation principle at a single time, the mean-field limit RRE \eqref{odex} is recovered using the concentration of measures with an explicit concentration rate (see Corollary \ref{cor:ode}).

Our other contribution is in constructing a stationary solution to the HJE. The construction of a stationary solution to the HJE is usually difficult and non-unique. However, under the assumption of the existence of a positive reversible invariant measure $\pi_{\V}$ satisfying \eqref{master_db_n}, we can construct an upper semicontinuous (USC) viscosity solution in the Barron-Jensen's sense \cite{Barron_Jensen_1990}; see Proposition \ref{prop:USC}. However, since the stationary HJE does not have uniqueness, whether our construction selects a meaningful weak KAM solution is still unknown. This selection principle for the drift-diffusion process is proved in \cite{GL23}.
}

While a comprehensive review of the vast literature on chemical reactions and the large deviation principle is beyond the scope of this work, let us review here some closely related works. The use of viscosity solutions to the Hamilton-Jacobi equation (HJE) as the limit of the nonlinear semigroup of Markov processes was developed by \textsc{Feng} and \textsc{Kurzt} in \cite{feng2006large} as a general framework to study the large deviation principle. See also some recent developments in \cite{Kraaij_2016, Kraaij_2020}. The idea of using the nonlinear semigroup and the variational principle of Markov processes to study the large deviation principle can be traced back to \cite{Varadhan_1966, fleming1983optimal}, while the idea of using viscosity solutions to HJE can be traced back to \cite{Ishii_Evans_1985, fleming1986pde, Barles_Perthame_1987}. Another general approach that uses HJE to study the large deviation principle in the physical literature is the macroscopic fluctuation theory developed by \textsc{Bertini}, et al. \cite{bertini2002macroscopic, Bertini15}; see further mathematical analysis and the variational structure in \cite{Renger_2018, Patterson_Renger_Sharma_2021}.
{\blue For chemical reactions, the sample path large deviation principle was first proved in \cite{Dembo18} via constructing a process with linear interpolation in time and the inverse contraction principle; and recent developments have addressed the boundary issue with a uniform vanishing rate at the boundary \cite{Agazzi_Andreis_Patterson_Renger_2021}.   The LDP in path space covers the single time LDP result in this paper, but the methodologies are completely different; see the remark after Corollary \ref{cor:ode}.}  A dynamic large deviation principle for the pair of concentrations and fluxes of chemical reaction jump processes was proved in \cite{Patterson_Renger_2019}. The connection between the generalized gradient flow and the good rate function in the large deviation principle was rigorously established by \textsc{Mielke, Renger, Peletier} \cite{Mielke_Renger_Peletier_2014}. The mean-field limit of the chemical reaction stochastic model was proved by Kurtz in \cite{kurtz1970solutions, Kurtz71}. Recently, \textsc{Maas} and \textsc{Mielke} \cite{maas2020modeling} provided another proof via the evolutionary $\Gamma$-convergence approach under the detailed balance assumption.

The remaining part of this paper is organized as follows. In Section \ref{sec2}, we revisit some terminologies of stochastic/deterministic chemical reaction equations and introduce the WKB reformulations for CME and the backward equation. In Section \ref{sec3}, we study the monotonicity, construct barriers to control the polynomial growth at far fields, and investigate the solvability of the associated monotone schemes (i.e., WKB reformulations). Then, in Section \ref{sec:ext_dd}, we prove the well-posedness of the backward equation, and in Section \ref{sec:db_d}, we recover the well-posedness of CME for reversible processes. In Section \ref{sec4}, we prove the convergence from the monotone schemes to the viscosity solution of the HJE. We discuss the construction of USC viscosity solutions for the stationary HJE in Section \ref{sec:sHJE} and for the dynamic HJE in Section \ref{sec:vis}, respectively. The short-time classical solution to the HJE and error estimates are discussed in Section \ref{sec5}. Finally, in Section \ref{sec6}, we use the convergence results, together with the exponential tightness, to prove the large deviation principle at a single time, which also recovers the mean-field limit RRE for chemical reactions.

\section{Preliminaries: stochastic and deterministic models for chemical reactions,  and WKB reformulations for forward/backward equations}\label{sec2}

In this section, we provide the necessary background for understanding the stochastic modeling of chemical reactions using the random time-changed Poisson process (see Section \ref{sec2.1}). We then derive CME \eqref{rp_eq} with the corresponding 'no reaction' boundary condition (see Section \ref{sec2.2}). Additionally, we revisit the macroscopic RRE \eqref{odex} in Section \ref{sec2.3}. In Sections \ref{sec2.4} and \ref{sec2.5}, we employ the WKB reformulation to transform the CME and the backward equation into nonlinear ODE systems, specifically nonlinear discrete semigroups on countable grids.

\subsection{Random time changed Poisson process and the law of mass action for chemical reactions}\label{sec2.1}

Chemical reactions involving $\ell=1,\cdots,N$ species $X_i$ and $j=1,\cdots,M$ reactions can be kinematically described as
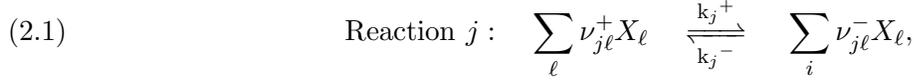
\begin{equation}\label{CRCR}
\text{Reaction }j: \quad \sum_{\ell} \nu_{j\ell}^+ X_{\ell} \quad \ce{<=>[k_j^+][k_j^-]} \quad   \sum_i \nu_{j\ell}^- X_\ell,
\end{equation}
where the nonnegative integers $\nu_{j\ell}^\pm \geq 0$ are stoichiometric coefficients and $k_j^\pm\geq 0$ are the reaction rates for the $j$-th forward and backward reactions. The column vector $\vec{\nu}_j:= \vec{\nu}_j^- - \vec{\nu}_j^+ := \bbs{\nu_{j\ell}^- - \nu_{j\ell}^+}_{\ell=1:N}\in \bZ^N$ is called the reaction vector for the $j$-th reaction, counting the net change in molecular numbers for species $X_\ell$. Let $\mathbb{N}$ be the set of natural numbers including zero. In this paper, all vectors $\vec{X}= \bbs{X_i}_{i=1:N} \in \mathbb{N}^N$ and $\bbs{\varphi_j}_{j=1:M},\, \bbs{k_j}_{j=1:M}\in \bR^M$ are column vectors.

Denote $\vec{m}=(m_\ell)_{\ell=1:N}$, where $m_\ell$ represents the molecular weight for the $\ell$-th species. Then the conservation of mass for the $j$-th reaction in \eqref{CRCR} implies 
\begin{equation}\label{mb_j}
\vec{\nu}_j \cdot \vec{m} = 0, \quad j=1,\cdots, M.
\end{equation}
Sometimes, in an open system where the materials exchange with the environment denoted as $\emptyset$, the exchange reaction $X_{\ell}\ce{<=>} \emptyset$ does not conserve mass.

Let $\mathbb{N}$ be the space of natural numbers, which serves as the state space for the counting process $X_\ell(t)$, representing the number of each species $\ell=1,\cdots,N$ in the biochemical reactions. With the reaction container size $1/h\gg 1$, the process $\cv_\ell(t)=hX_\ell(t)$ satisfies the random time-changed Poisson representation for chemical reactions \eqref{Csde} (see \cite{kurtz1980representations, Kurtz15}):
\begin{align*}
\cv(t) = \cv(0) + \sum_{j=1}^M  \vec{\nu}_{j} h \Bigg(\mathbbm{1}_{\{\cv(t^-)+\vec{\nu}_j h \geq 0\}} Y^+_j  \bbs{\frac{1}{h}\int_0^t \tilde{\Phi}^+_j(\cv(s))\ud s} \\
-\mathbbm{1}_{\{\cv(t^-)-\vec{\nu}_j h \geq 0\}}Y^-_j  \bbs{\frac{1}{h}\int_0^t \tilde{\Phi}^-_j(\cv(s))\ud s}\Bigg).
\end{align*}
Here, for the $j$-th reaction channel, $Y^\pm_{j}(t)$ are i.i.d. unit-rate Poisson processes, and the intensity function is given by the mesoscopic law of mass action (LMA):
\begin{equation}\label{newR}
\tilde{\Phi}_j^\pm(\cv) = {k_j^\pm}  \prod_{\ell=1}^{N} \frac{\bbs{\frac{\cv_\ell}{h}} ! \ h^{\nu_{j\ell}^\pm}}{ \bbs{\frac{\cv_\ell}{h} - \nu_{j\ell}^\pm}!}. 
\end{equation}
The `no reaction' constraints $\cv(t^-)\pm \vec{\nu}_{j} h\geq 0$ ensure that no jump occurs if the number of some species would become negative in the container. This `no reaction' correction to the process \eqref{Csde} was also noticed in \cite[eq(28)]{anderson2019constrained}, where a very similar `no reaction' constraint was imposed near the relative boundary of the positive orthant. Here, we use $\tilde{\Phi}^\pm_j$ to distinguish the mesoscopic LMA from the limit macroscopic LMA $\Phi^\pm_j$ in \eqref{lma}. In the literature, this process \eqref{Csde} is also known as the Marcus-Lushnikov process \cite{marcus1968stochastic, lushnikov1978coagulation} or Gillespie's process \cite{gillespie1972stochastic}. The existence and uniqueness of the stochastic equation \eqref{Csde} were proved by \cite{kurtz1980representations, Kurtz15} in terms of the corresponding martingale problem.

For the purpose of effectively handling the boundary condition when some species become zero, we extend the process $\cv(t)$ to include negative counting numbers for species such that those negative numbers remain unchanged. Clearly, if $\cv_\ell(0)<0$ for some $\ell$, then $\cv(t) \equiv \cv(0)$. To realize this, we also set the LMA as:
\begin{equation}\label{exPHI}
\tilde{\Phi}_j^\pm(\cv) = 0 \quad \text{if for some $\ell$, $\cv_\ell<0$}.
\end{equation}
In summary, the original process is given by:
\begin{equation}
\cv \in \Omega^+_{\V} := \{ \vxi = \vec{i} h; \vec{i} \in \mathbb{N}^N \},
\end{equation}
while after including the above extended notions, the process \begin{equation}
\cv\in \Omega_{\V}= \{ \vxi=\vec{i} h; \vec{i}\in \mathbb{Z}^N \}
\end{equation}
 can still be described by \eqref{Csde}.
In other words, if $\cv(0) \in \Omega_{\V} \backslash \Omega_{\V}^+$, then $\cv(t) \equiv \cv(0)$.

\subsection{Chemical master equation and `no reaction' boundary condition}\label{sec2.2}

For a chemical reaction modeled by \eqref{Csde}, we denote the counting probability of $\cv(t)$ as $p_{\V}(\vxv,t) = \bE(\mathbbm{1}_{\vxv}(\cv(t)))$, where $\mathbbm{1}_{\vxv}$ is the indicator function. Then $p_{\V}(\vxv,t)$, $\vxi\in \Omega_h$, satisfies  CME \cite{Kurtz15}:
\begin{equation}\label{rp_eq}
\begin{aligned}
\frac{\ud}{\ud t} p_{\V}(\vxi, t) =& \frac{1}{h}\sum_{j=1, \vxi- \vec{\nu}_{j}h\geq 0}^M  \left( \tilde{\Phi}^+_j(\vxi- \vec{\nu}_{j} h) p_{\V}(\vxi- \vec{\nu}_{j} h,t) - \tilde{\Phi}^-_j(\vxi) p_{\V}(\vxi,t) \right) \\
& + \frac{1}{h}\sum_{j=1, \vxi+\vec{\nu}_{j} h\geq 0}^M \left( \tilde{\Phi}^-_j(\vxi+\vec{\nu}_{j}h) p_{\V}(\vxi+\vec{\nu}_{j}h,t) - \tilde{\Phi}_j^+(\vxi) p_{\V}(\vxi,t) \right), \quad \text{for } \vxi\in \omp;\\ 
\frac{\ud}{\ud t} p_{\V}(\vxi, t) =& 0, \quad \text{for } \vxi \in \omn.
\end{aligned}
\end{equation}
It is natural to take $p_{\V}(\vxi, 0) = 0$ for $\vxi\in\omn$, which remains $0$ for all times for $\vxi\in\omn$. In the $Q$-matrix form, we denote $\frac{\ud}{\ud t} p_{\V} = Q^*_{\V} p_{\V}$, where $Q^*_{\V}$ is the transpose of the generator $Q_{\V}$ of process $\cv(t)$. For the derivation of the generator $Q_{\V}$ with the `no reaction' constraints, we refer to \cite[Appendix]{GL22}.

Here, the constraints $\vxi\pm \vec{\nu}_{j} h\geq 0$ are understood componentwisely, and we refer to this as the `no reaction' boundary condition for  CME. The `no reaction' boundary condition implies that if the number of some species is negative in the container, then no chemical reaction can occur.
In Lemma \ref{lem:mass1}, we will show that under this boundary constraint, the total probability is conserved. In Lemma \ref{lem:generator}, we will derive the generator $Q_{\V}$ and the backward equation associated with the boundary constraint.

In the literature, a commonly used simple Dirichlet boundary condition is $p_{\V}(\vxi) = 0$ for $\vxi\in \omn$ (or equivalently $\tilde{\Phi}^\pm(\vxi) = 0$ for $\vxi\in\omn$), as seen in \cite{Kurtz15, Gauckler14, patterson2019large, maas2020modeling}. However, the `no reaction' boundary condition is directly derived from \eqref{Csde} and imposes the condition that any reaction jump resulting in a negative species count cannot occur. This is in the spirit of the Feller-Wentzell type dynamic boundary conditions proposed by Feller and Wentzell \cite{Feller_1952, Feller_1954, venttsel1959boundary} to preserve the total probability. A similar boundary condition was used in \cite{anderson2019constrained}, where a constrained Langevin approximation was established to incorporate the boundary condition.

An immediate lemma below shows the conservation of total probability.

\begin{lem}\label{lem:mass1}
For CME \eqref{rp_eq}, we have the conservation of total probability:
\begin{equation}
\frac{\ud }{\ud t}  \sum_{\vxi \in \Omega_{\V}}p_{\V}(\vxi, t) = \frac{\ud }{\ud t}  \sum_{\vxi \in\omp}p_{\V}(\vxi, t)  = 0.
\end{equation}
\end{lem}

\begin{proof}
From \eqref{rp_eq}, we have
\begin{align}\label{mass1}
 \frac{\ud }{\ud t} \sum_{\vxi \geq 0}p_{\V}(\vxi, t) 
  =& \frac{1}{h}\sum_{\vxi \geq 0}\sum_{j=1, \vxi-\vec{\nu}_{j}h\geq 0}^M   \tilde{\Phi}^+_j(\vxi-\vec{\nu}_{j}h)     p_{\V}(\vxi-\vec{\nu}_{j}h,t) -      \tilde{\Phi}^-_j(\vxi)     p_{\V}(\vxi,t) \nonumber \\
 &  +  \frac{1}{h}\sum_{\vxi \geq 0}\sum_{j=1, \vxi+ \vec{\nu}_{j}h\geq 0}^M   \tilde{\Phi}^-_j(\vxi+ \vec{\nu}_{j}h)     p_{\V}(\vxi+ \vec{\nu}_{j}h,t) -  \tilde{\Phi}_j^+(\vxi) p_{\V}(\vxi,t).
\end{align}
Then changing the variable $\vyi = \vxi+ \vec{\nu}_{j}h$, the second line above becomes
\begin{equation}
\frac{1}{h}\sum_{\vyi\geq 0}\sum_{\vyi - \vec{\nu}_{j} h \geq 0}   \tilde{\Phi}^-_j(\vyi)     p_{\V}(\vyi,t) -  \tilde{\Phi}_j^+(\vyi -  \vec{\nu}_{j}h) p_{\V}(\vyi -  \vec{\nu}_{j}h,t).
\end{equation}
And thus $\frac{\ud }{\ud t}  \sum_{\vxi \geq 0}p_{\V}(\vxi, t)  = 0$.

\end{proof}

The rigorous justification of non-explosion $\sum_{\vxi} p_{\V}(\vxi) = 1$ for CME was proved in \cite{maas2020modeling} under the detailed balance assumption \eqref{master_db_n}.
 We refer to \cite{Gauckler14} for the existence and regularity of solutions to CME.

\subsection{Mean-field limit is the macroscopic reaction rate equation}\label{sec2.3}

The mesoscopic jumping process $\cv$ in \eqref{Csde} can be regarded as a large-size interacting particle system. In the large-size limit (thermodynamic limit), this interacting particle system can be approximately described by a mean field equation, i.e., a macroscopic nonlinear chemical reaction-rate equation. If the law of large numbers in the mean field limit holds, i.e., $p_{\V}(\vxv, t)\to \delta_{\vx(t)}$ for some $\vx(t)$, then the limit $\vx(t)$ describes the dynamics of the concentration of $N$ species in the continuous state space $\mathbb{R}_+^N := \{\vx \in \bR^N; x_\ell > 0\}$ and is given by  RRE \eqref{odex}, also known as the chemical kinetic rate equation. The macroscopic fluxes $\Phi_j^\pm$ satisfy the macroscopic LMA:
\begin{equation}\label{lma}
\Phi^\pm_j(\vx) = k^\pm_j  \prod_{\ell=1}^{N} \bbs{x_\ell}^{\nu^\pm_{j\ell}}.
\end{equation}
For simplicity of analysis, we also extend this LMA as:
\begin{equation}\label{exLMA}
\Phi^{\pm}_j(\vx) = 0 \quad \text{if } \vx \in \bR^N \backslash \bR^N_+.
\end{equation}
This is a macroscopic approximation for the mesoscopic LMA \eqref{newR} since $\tilde{\Phi}_j^\pm(\vxi) \to \Phi_j^\pm(\vx)$ as $\vxi \to \vx$. This RRE with LMA was first proposed by Guldberg and Waage in 1864. The detailed balance condition for RRE \eqref{odex} is defined by Wegscheider (1901) and Lewis (1925): there exists $\peq > 0$ (componentwise) such that
\begin{equation}\label{DB}
\Phi^+_j(\peq) - \Phi^-_j(\peq) = 0, \quad \forall j.
\end{equation}

Kurtz \cite{kurtz1970solutions, Kurtz71} proved the law of large numbers for the large-size process $\cv(t)$; cf. \cite[Theorem 4.1]{Kurtz15}. That is, if $\cv(0)\to \vec{x}(0)$ as $h\to 0$, then for any $\epsilon > 0$ and $t > 0$,
\begin{equation}\label{lln}
\lim_{h \to 0} \bP\{ \sup_{0\leq s\leq t}|\cv(s)-\vec{x}(s)|\geq \epsilon \}=0.
\end{equation}
Thus, we will also refer to the large-size limiting ODE \eqref{odex} as the macroscopic RRE. This provides a passage from the mesoscopic LMA \eqref{newR} to the macroscopic one \eqref{lma}. {\blue Recent results in \cite{maas2020modeling} establish the evolutionary $\Gamma$-convergence from CME to the Liouville equation for the case that both CME and the limiting equation has a generalized gradient flow structure. Starting from a deterministic state $\vx_0$, \cite[Theorem 4.7]{maas2020modeling} recovers Kurtz's results on the mean field limit of CME. However, the approach in \cite{{maas2020modeling}}  requires the detailed balance assumption.} For the derivation of the mean field limit of CME with 'no reaction' constraints, please refer to \cite[Appendix]{GL22}.

\subsection{WKB  reformulation for CME and discrete HJE}\label{sec2.4}

Besides the macroscopic trajectory $\vx(t)$ given by the law of large numbers, the WKB expansion for $p_{\V}(\vxv, t)$ in CME \eqref{rp_eq} is another standard method \cite{Kubo73, Hu87, Dykman_Mori_Ross_Hunt_1994, SWbook95, QianGe17}, which builds a more informative bridge between mesoscopic dynamics and macroscopic behaviors.
To characterize the exponential asymptotic behavior, we make a change of variable
\begin{equation}\label{1.3}
 p_{\V}(\vxv,t) =  e^{-\frac{\psi_{\V}(\vxv,t)}{h}}, \quad p_{\V}(\vxv,0)=p_0(\vxv).
\end{equation}
CME can be recast as the following nonlinear ODE system for $\vxi \in \omp$:
\begin{equation}\label{upwind00}
\begin{aligned}
\partial_t \psi_{\text{h}}(\vxi) + \sum_{j=1, \,\vxi- \vec{\nu}_{j}h\geq 0}^M  &\tilde{\Phi}^+_j(\vxi-\vec{\nu}_{j}h )    e^{\left(\frac{\psi_{\V}(\vxi) - \psi_{\V}(\vxi-\vec{\nu}_jh)}{h}\right)}
-    \tilde{\Phi}_j^-(\vxi)  \\   
&+ \sum_{j=1, \,\vxi+\vec{\nu}_{j}h\geq 0}^M \tilde{\Phi}^-_j(\vxi+ \vec{\nu}_{j}h)     e^{\left(\frac{\psi_{\V}(\vxi)-\psi_{\V}(\vxi+\vec{\nu}_{j}h)}{h}\right)}-      \tilde{\Phi}^+_j(\vxi)=0,
\end{aligned}
\end{equation}
while for $\vxi\in\omn$, $\partial_t \psi_{\text{h}}(\vxi)=0$.
We point out that, in general, a constant is not a solution to \eqref{upwind00}. However, if $\psi_{\V}$ is a solution, then $\psi_{\V}+c$ is also a solution to \eqref{upwind00}.
Formally, as $h\to 0$, Taylor's expansion of $\psi_{\V}$ with respect to $h$ leads to the following HJE for the rescaled master equation \eqref{rp_eq} for $\psi(\vx,t)$:
\begin{equation}\label{HJEpsi}
  \pt_t \psi(\vec{x}, t) =   -\sum_{j=1}^M     \bbs{ \Phi^+_j(\vec{x})\bbs{e^{\vec{\nu_j} \cdot \nabla \psi(\vec{x},t)}   -  1} +  \Phi^-_j(\vec{x})\bbs{e^{-\vec{\nu_j} \cdot \nabla \psi(\vec{x},t)}   - 1 }}.
 \end{equation}

Define the Hamiltonian $H(\vp,\vx)$ on $\mathbb{R}^N\times \mathbb{R}^N$ as follows. For $\vx\in \mathbb{R}^N_+$,
\begin{equation}\label{H}
H(\vec{p},\vec{x}) := \sum_{j=1}^M \left[ \Phi^+_j(\vec{x})e^{\vec{\nu_j} \cdot \vec{p}} - \Phi^+_j(\vec{x}) + \Phi^-_j(\vec{x})e^{-\vec{\nu_j} \cdot\vec{p}} - \Phi^-_j(\vec{x}) \right].
\end{equation}
Recall that in \eqref{exLMA}, $\Phi^{\pm}_j(\vx)=0$ for $\vx\in \mathbb{R}^N\setminus \mathbb{R}^N_+$. We naturally extend $H(\vp,\vx)$ to $\mathbb{R}^N\times \mathbb{R}^N$ such that $H(\vp,\vx)=0$ for $\vx\notin \mathbb{R}^N_+$.
Then the HJE for $\psi(\vx,t)$ can be recast as
\begin{equation}\label{HJE2psi00}
\partial_t \psi + H(\nabla \psi, \vx) =0, \quad \vx\in\mathbb{R}^N.
\end{equation}
The WKB analysis above defines a Hamiltonian $H(\vp,\vx)$, which contains almost all the information for the macroscopic dynamics \cite{Dykman_Mori_Ross_Hunt_1994, GL22}. For this reason, we also call \eqref{upwind00} the CME in the HJE form.
We remark that this kind of WKB expansion was first used by Kubo et al. \cite{Kubo73} for master equations of general Markov processes and later applied to CME in \cite{Hu87}. In \cite{Dykman_Mori_Ross_Hunt_1994}, the HJE \eqref{HJE2psi00} with the associated Hamiltonian $H$ in \eqref{H} was formally derived.
{\blue In the mathematical analysis later, we will impose a technique assumption that $\Phi^\pm_j(\vx)$ is locally  Lipschitz continuous after the zero extension. We point out this assumption might be removed by directly consider the optimal control problem in a domain with a boundary and impose an appropriate boundary cost, which will be an interesting future study. }

\subsection{WKB reformulation for the backward equation is Varadhan's  nonlinear semigroup}\label{sec2.5}

The fluctuation on path space, i.e., the large deviation principle, can be computed through WKB expansion for the backward equation.
Recall the rescaled process $\cv(t)$ in \eqref{Csde}, which is also denoted as $\cv_t$ for simplicity.
For any $f\in C_b(\mathbb{R}^N_+)$, denote 
\begin{equation}
w_{\V}(\vxv, t) := \mathbb{E}^{\vxv}[f(\cv_t)],
\end{equation}
then $w_{\V}(\vxv,t)$ satisfies the backward equation
\begin{equation}\label{backward}
\partial_t w_{\V} = Q_{\V} w_{\V}, \quad w_{\V}(\vxv, 0)=f(\vxv).
\end{equation}
Notice the 'no reaction' boundary condition we derived in CME \eqref{rp_eq}. Below we give a lemma to explicitly derive the associated boundary condition in the generator $Q_{\V}$ as the duality of $Q^*_{\V}$. 
\begin{lem}\label{lem:generator}
The backward equation with explicit generator $Q_{\V}$ can be expressed as the duality of the forward equation
\begin{equation}
\begin{aligned}
\partial_t w_{\V}(\vxi,t) =& \frac{1}{h}\sum_{j=1, \vxi+\vec{\nu_{j}} h\geq 0}^M \tilde{\Phi}^+_j(\vxv)\left[ w_{\V}(\vxv+ \vec{\nu_{j}}h)  - w_{\V}(\vxv) \right] \\
&+ \frac{1}{h}\sum_{j=1, \vxi-\vec{\nu_{j}} h\geq 0}^M \tilde{\Phi}_j^-(\vxv)\left[ w_{\V}(\vxv- \vec{\nu_{j}}h )-w_{\V}(\vxv) \right], \quad \vxi\in\omp; \\
\partial_t w_{\V}(\vxi,t) =& 0, \quad \vxi\in \omn.
\end{aligned}
\end{equation}
\end{lem}

\begin{proof}
Multiplying \eqref{rp_eq} by $w_{\V}$ and summation yields
\begin{equation}
\begin{aligned}
\langle w_{\V}, Q_{\V}^* p \rangle = & \frac{1}{h}\sum_{\vxi\geq 0}\sum_{j=1, \vxi- \vec{\nu}_{j}h\geq 0}^M \tilde{\Phi}^+_j(\vxi- \vec{\nu}_{j} h) p_{\V}(\vxi- \vec{\nu}_{j} h,t)w_{\V}(\vxi) - \tilde{\Phi}^-_j(\vxi) p_{\V}(\vxi,t)w_{\V}(\vxi) \\
& + \frac{1}{h}\sum_{\vxi\geq 0}\sum_{j=1, \vxi+\vec{\nu}_{j} h\geq 0}^M \tilde{\Phi}^-_j(\vxi+\vec{\nu}_{j}h) p_{\V}(\vxi+\vec{\nu}_{j}h,t)w_{\V}(\vxi) - \tilde{\Phi}_j^+(\vxi) p_{\V}(\vxi,t)w_{\V}(\vxi).
\end{aligned}
\end{equation}
Here from \eqref{rp_eq}, we used 
\begin{equation}
\langle w_{\V}, Q_{\V}^* p \rangle = \langle w_{\V}, Q_{\V}^* p \rangle_{\omp}.
\end{equation}
Changing variable $\vyi=\vxi- \vec{\nu}_{j}h$ in the above $\tilde{\Phi}^+(\vxi- \vec{\nu}_{j}h)$ term while changing variable $\vyi=\vxi+ \vec{\nu}_{j}h$ in the above $\tilde{\Phi}^-(\vxi+ \vec{\nu}_{j}h)$, we have
\begin{equation}
\begin{aligned}
\langle Q_{\V} w_{\V}, p \rangle_{\omp} = & \frac{1}{h}\sum_{\vyi\geq 0}\sum_{j=1, \vyi+ \vec{\nu}_{j}h\geq 0}^M \tilde{\Phi}^+_j(\vyi) p_{\V}(\vyi ,t)[w_{\V}(\vyi+ \vec{\nu}_{j} h) - w_{\V}(\vyi)] \\
& + \frac{1}{h}\sum_{\vyi\geq 0}\sum_{j=1, \vyi-\vec{\nu}_{j} h\geq 0}^M \tilde{\Phi}^-_j(\vyi ) p_{\V}(\vyi,t)[w_{\V}(\vyi- \vec{\nu}_{j}h) - w_{\V}(\vyi)].
\end{aligned}
\end{equation}
Meanwhile, we set
\begin{equation}
Q_{\V} w_{\V} = 0, \quad \vxi\in \omn.
\end{equation}
Regarding $p$ as arbitrary test functions, this gives the backward equation \eqref{backward}.
\end{proof}

Denote $\uh(\vxi,t)= h\log w_{\V}(\vxi,t)$ for $\vxi\in\Omega_{\V}$.
We obtain a nonlinear ODE system for the backward equation in HJE form
\begin{equation}\label{upwind0}
\pt_t \uh(\vxv,t) = h e^{- \frac{\uh(\vxv,t)}{h}} Q_{\V} e^{ \frac{\uh(\vxv,t)}{h}} =:  H_{h}( \uh), \quad \uh(\vxv,0) = h\log f(\vxv).
\end{equation}
The above WKB reformulation for the backward equation is equivalent to
\begin{equation}\label{wkb_b}
\uh(\vxv,t) =h \log w_{\V}(\vxv,t) =  h \log \bE^{\vxv}\bbs{f(\cv_t)} = h \log \bE^{\vxv}\bbs{e^{ \frac{u_0(\cv_t)}{h}}} =:\bbs{S_t u_0}(\vxv),
\end{equation}
which is the so-called nonlinear semigroup \cite{Varadhan_1966, feng2006large} for process $\cv(t)$. Notice here the process $\cv$ and the nonlinear semigroup are all extended to $\vxi\in\Omega_{\V}$.

We first write  $H_{\V}$  explicitly. For $\vxi\geq 0$,
\begin{align}\label{tm24}
  H_{\V}(u) =&  \sum_{j=1, \vxi+\vec{\nu_{j}} h\geq 0}^M     \tilde{\Phi}^+_j(\vxv)\bbs{ e^{\frac{\uh(\vxv+ \vec{\nu_{j}} h)-\uh(\vxv)}{h}} - 1}    + \sum_{j=1, \vxi-\vec{\nu_{j}} h\geq 0}^M \tilde{\Phi}_j^-(\vxv)\bbs{ e^{  \frac{\uh(\vxv- \vec{\nu_{j}}h)-\uh(\vxv)}{h} } - 1}. 
\end{align}
From the extended $\tilde{\Phi}^{\pm}_j(\vxi)$ in \eqref{exPHI}, \eqref{tm24} also naturally extended to $\vxi\in\Omega_{\V}$.
We see $\pt_t \uh = H_{\V}(\uh)$ is a nonlinear ODE system on the countable grids.
We point out a constant is always a solution to \eqref{upwind0}. Meanwhile, if $\uh$ is a solution then $\uh+c$ is also a solution to \eqref{upwind0}. We will use these two important invariance facts to prove the existence and contraction principle later; see Lemma \ref{lem:perron} and Proposition \ref{prop:cp}.

Take a formal limit in \eqref{upwind0} with the discrete Hamiltonian $H_{\V}$ in \eqref{tm24}, we obtain the HJE in terms of $u(\vx,t), \,\vx\in \mathbb{R}^N$, i.e., \eqref{HJE2psi}
\begin{equation}
\partial_t u(\vx,t) =  \sum_{j=1}^M \left[ \Phi^+_j(\vec{x})\left(e^{\vec{\nu}_j\cdot \nabla u(\vx, t)}-1\right)     +   \Phi_j^-(\vec{x})\left( e^{-\vec{\nu}_j\cdot \nabla u(\vx, t)}-1\right)  \right] = H(\nabla u(\vx), \vx).
\end{equation}
We remark that the duality between the forward and backward equations, after the WKB limit, leads to the same HJE with only a sign difference in the time derivative.
The boundary condition for \eqref{upwind0} is incorporated into the constraint $\vxi \pm \vec{\nu_{j}} h\geq 0$ in the equation.
Imposing a physically meaningful boundary condition for the limit HJE \eqref{HJE2psi} at $\vx=\vec{0}$ is complicated. However, in our construction of the viscosity solution to the HJE, as a limit of \eqref{upwind0}, it automatically inherits the boundary condition of \eqref{upwind0}, and at the same time, the HJE \eqref{HJE2psi} is defined on the whole space $\mathbb{R}^N$ by using an extension; see Section \ref{sec:vis}. This is in the same spirit of the Feller-Wentzell type dynamic boundary conditions, which were proposed by Feller and Wentzell to impose proper boundary conditions preserving total mass \cite{Feller_1952, Feller_1954, venttsel1959boundary}.

From now on, for both the discrete monotone scheme \eqref{upwind} and the HJE \eqref{HJEpsi}, the domain refers to the extended domain $\Omega_{\V}$ and $\mathbb{R}^N$, respectively.

 \section{Wellposedness of CME and backward equation via nonlinear semigroup}\label{sec3}

In this section, we study the well-posedness of CME \eqref{rp_eq} and the backward equation \eqref{backward} via the nonlinear semigroup method. First, the WKB reformulation for both CME \eqref{rp_eq} and the backward equation \eqref{backward} leads to two discrete HJEs, respectively. The main difference between them is the sign of the time derivative. 

Next, we observe that both CME and the backward equation can be regarded as monotone schemes for linear hyperbolic equations. These monotone schemes, after the WKB reformulation, are still monotone schemes for the corresponding HJEs. The essential features of these monotone schemes are the monotonicity and the nonexpansive estimates, which naturally generate a nonlinear semigroup thanks to Crandall-Liggett's nonlinear semigroup theory. 
Using this observation, this section obtains a unique discrete solution to the corresponding HJEs.

In Section \ref{sec:upwind}, we study the monotone and nonexpansive properties of the monotone scheme with backward Euler time discretization and then prove the existence and uniqueness of the backward Euler scheme \eqref{bEuler} by the Perron method. 
We point out that we cannot directly use the explicit scheme in \cite{crandall1984two} because LMA \eqref{lma} leads to a polynomial growth in the coefficients $\Phi^\pm_j(\vx)$. Thus, we will prove the existence and uniqueness of this backward Euler scheme \eqref{bEuler}. 
In Section \ref{sec:ext_dd}, by taking $\Delta t\to 0$ and using the nonexpansive nonlinear semigroup constructed by Crandall-Liggett \cite{CrandallLiggett}, we prove the global existence of the solution to the backward equation \eqref{backward}. 

In order to recover the existence and exponential tightness of CME from the results of the backward equation, Section \ref{sec:db_d} assumes the existence of a positive reversible invariant measure $\pi_{\V}$ satisfying \eqref{master_db_n} and then proves the existence, comparison principle, and exponential tightness of CME; see Corollary \ref{cor:tight}. This reversibility assumption includes  some non-equilibrium enzyme reactions \cite{GL22}, where the energy landscape (i.e., the limiting solution to the stationary HJE in Proposition \ref{prop:USC}) is nonconvex with multiple steady states indicating three key features of non-equilibrium reactions: multiple steady states, non-zero flux, and positive entropy production rate at the non-equilibrium steady states (NESS).

\subsection{A monotone scheme for HJE inherited from CME}\label{sec:upwind}
Recall the backward equation \eqref{backward} and the WKB reformulation \eqref{upwind0}. Recall the countable discrete domain as
\begin{equation}
\Omega_{\V}:= \{ \vxi=\vec{i} h; \vec{i}\in \mathbb{Z}^N \}.
\end{equation}
\eqref{upwind} is exactly a monotone scheme for the corresponding HJE \eqref{HJE2psi}.

We remark that the solution to \eqref{upwind} shifted by a constant is still a solution. After proving the well-posedness, we also refer to the solution to \eqref{upwind} as Varadhan's nonlinear semigroup.

\subsubsection{Nonexpansive property for backward Euler monotone scheme}

We now rewrite the monotone scheme \eqref{upwind} following the \textsc{Barles-Souganidis} framework \cite{barles1991convergence}. Denote the monotone scheme operator as
\begin{equation}\label{def_H}
H_{\V}: D(H_{\V}) \subset \ell^{\infty}(\Omega_{\V}) \to \ell^{\infty}(\Omega_{\V}); \qquad  ( u(\vxi) )_{\vxi\in\Omega_{\V}} \mapsto ( H_{\V}(\vxi, u(\vxi), u(\cdot)) )_{\vxi\in\Omega_{\V}},
\end{equation}
where $D(H_{\V})$ is the domain of $H_{\V}$ defined as $D(H_{\V}):=\{(u(\vxi))\in\ell^{\infty}(\Omega_{\V}); (H_{\V}(\vxi, u(\vxi), u))\in \ell^{\infty}(\Omega_{\V})\}$.

The discrete Hamiltonian $H_{h}$ is defined as
\begin{equation}\label{def_Hh}
H_{h}(\vxi, \psih(\vxi), \varphi):= \sum_{j=1, \,\vxi+ \vec{\nu}_{j}h\geq 0}^M  \tilde{\Phi}^+_j(\vxv)\left( e^{\frac{\varphi(\vxv+ \vec{\nu_{j}} h)-\uh(\vxv)}{h}} - 1 \right) + \sum_{j=1, \,\vxi-\vec{\nu}_{j}h\geq 0}^M \tilde{\Phi}_j^-(\vxv)\left( e^{  \frac{\varphi(\vxv- \vec{\nu_{j}}h)-\uh(\vxv)}{h} } - 1 \right).
\end{equation}
This notation of the discrete Hamiltonian follows \cite{barles1991convergence}.

It is straightforward to verify that $H_{h}$ satisfies the monotone scheme condition, i.e.,
\begin{equation}\label{mono}
\text{if } \varphi_1 \leq \varphi_2 \text{, then } H_{h}(\vx, \uh(\vx), \varphi_1) \leq H_{h}(\vx,\uh(\vx), \varphi_2).
\end{equation}

Denote the time step as $\Delta t$. Then the backward Euler time discretization gives 
\begin{equation}\label{bEuler}
\frac{\psih^{n}-\psih^{n-1}}{\Delta t} - H_{\V}(\vxi, \psih^{n}(\vxi), \psih^{n})=0.
\end{equation}
Here, the second variable in $H$ is the value $\psih^n(\vxi)$ at $\vxi$, and the third variable $\psih$ replaces $\varphi$ in \eqref{def_H}, indicating the dependence of the values $\psih$ at the surrounding jumping points $\psih(\vxi\pm \vec{\nu}_j h)$. This notation clearly shows that $H_{\V}$ is decreasing with respect to the second variable while increasing with respect to the third variable.
Then $\psih^{n}=(I-\Delta t H_{\V})^{-1} \psih^{n-1}:= J_{\Delta t, h} \psih^{n-1}$ can be solved as the solution $u_{\V}$ to the discrete resolvent problem with $f_{\V}=\psih^{n-1}$
\begin{equation}\label{dis_resolvent}
u_{\V}(\vxi) - \Delta t H_{\V}(\vxi, u_{\V}(\vxi), u_{\V})=f_{\V}(\vxi).
\end{equation}

We first prove the nonexpansive property for the discrete solutions, which implies the uniqueness of the solution to \eqref{bEuler}.
\begin{lem}\label{lem_nonexp_d}
Let $u_{\V}$ and $v_{\V}$ be two solutions satisfying $u_{\V}= J_{\Delta t, h} f_{\V}$ and $v_{\V}=J_{\Delta t, h} g_{\V}$. Then we have
\begin{enumerate}[(i)]
\item Monotonicity: if $f_{\V}\leq g_{\V}$, then $u_{\V}\leq v_{\V}$;
\item Nonexpansive: 
\begin{equation}\label{nonexp_d}
\|u_{\V}-v_{\V}\|_{\infty} \leq \|f_{\V}-g_{\V}\|_{\infty}.
\end{equation}
\end{enumerate}
\end{lem}

\begin{proof}
First, denote $\vxi^*$ as a maximum point such that $\max(u_{\V}-v_{\V})=(u_{\V}-v_{\V})(\vxi^*)$. Here, without loss of generality, we assume the maximum is attained; otherwise, the proof is still true by passing to a limit. Then
\begin{equation}\label{mp1}
\begin{aligned}
 (u_{\V}-v_{\V})(\vxi^*)  = & \frac{1}{h}\sum_{j=1, \,\vxi+ \vec{\nu}_{j}h\geq 0}^M \tilde{\Phi}^+_j(\vxi^* ) e^{\xi_i} [(u_{\V}-v_{\V})(\vxi^*+ {\vec{\nu}_j}h) - (u_{\V}-v_{\V})(\vxi^* )] \\
  & + \frac{1}{h}\sum_{j=1, \,\vxi-\vec{\nu}_{j}h\geq 0}^M \tilde{\Phi}^-_j(\vxi^* ) e^{\xi_i'} [(u_{\V}-v_{\V})(\vxi^*-{\vec{\nu}_j}h)-(u_{\V}-v_{\V})(\vxi^*)] + (f_{\V}-g_{\V})(\vxi^*)\\
   \leq& (f_{\V}-g_{\V})(\vxi^*),
\end{aligned}
\end{equation}
where $\xi_i, \xi_i'$ are some constants from the mean value theorem.
This immediately gives
\begin{equation}
\max (u_{\V}-v_{\V}) \leq \max(f_{\V}-g_{\V}).
\end{equation}

Then if $f_{\V}\leq g_{\V}$, we have $u_{\V}\leq v_{\V}$, which concludes (i).
Notice also the identity
\begin{equation}
\|u_{\V}-v_{\V}\|_{\infty} = \max\{ \max(u_{\V}-v_{\V}), \max(v_{\V}-u_{\V})\} \leq \max\{ \max(f_{\V}-g_{\V}), \max(f_{\V}-g_{\V})\}=\|f_{\V}-g_{\V}\|_{\infty},
\end{equation}
which concludes (ii).
\end{proof}

The property (ii) shows that the resolvent $J_{\Delta t, h} = (I - \Delta t H_{\V})^{-1}$ is nonexpansive. Thus, we also obtain that $-H_{\V}$ is an accretive operator on $\ell^\infty(\Omega_{\V}).$

\subsubsection{Existence and uniqueness of the backward Euler scheme via the Perron method}

We next prove the following lemma to ensure the existence and uniqueness of a solution $\psih^{n}$ to \eqref{bEuler}, i.e., solving the resolvent problem \eqref{dis_resolvent}. The proof of this lemma follows the classical Perron method and the observation that a constant is always a solution to \eqref{upwind}. 
\begin{lem}\label{lem:perron}
Let $\lambda>0$. Assume there exist constants $K_M:=\sup_i f(\vxi)$ and $K_m:= \inf_i f(\vxi)$. Then there exists a unique solution $\uh$ to
\begin{equation}\label{dis_exist}
\uh(\vxi) - \lambda  H_{\V}(\vxi, \uh(\vxi), \uh) = f(\vxi), \quad \vxi\in\Omega_{\V}.
\end{equation}
We also have $K_m\leq \psih \leq K_M$.
\end{lem}
\begin{proof}
Step 1. We define a discrete subsolution $\uh$ to \eqref{dis_exist} as
\begin{equation}
  \uh(\vxi) - \lambda  H_{\V}(\vxi, \uh(\vxi), \uh) \leq  f(\vxi), \quad \vxi\in\Omega_{\V}.
\end{equation}
Since $H_{\V}$ satisfies the monotone scheme condition \eqref{mono}, we can directly verify that the above discrete subsolution definition is consistent with the viscosity subsolution; see the definition in Appendix \ref{app3}. Supersolution is defined by reversing the inequality.

Based on this definition, we can directly verify that $K_M$ is a supersolution while $K_m$ is a subsolution.
Denote
\begin{equation}
\mathcal{F}:=\{v; \quad v \text{ is a subsolution to \eqref{dis_exist}}, \, v\leq K_M\}.
\end{equation}
Since $K_m\in\mathcal{F}$, $\mathcal{F}\neq \emptyset$.

Step 2. Define
\begin{equation}\label{def_u}
\uh(\vxi) := \sup \{ \vh(\vxi); \,\, v\in \mathcal{F}\}.
\end{equation}
We will prove that $\uh(\vxi)$ is a subsolution.
It is obvious that $\uh\geq K_m$. Since $\Omega_{\V}$ is countable, by the diagonal argument, there exists a subsequence $v_k\in\mathcal{F}$ such that for any $\vxi$, $v_k(\vxi) \to \uh(\vxi)$ as $k\to+\infty$.
Then, from the continuity of $H_{\V}$, for each $\vxi\in \Omega_{\V}$, taking $k\to +\infty$, we have
\begin{equation}\label{217}
\uh(\vxi)- \lambda  H_{\V}(\vxi, \uh(\vxi), \uh) \leq  f(\vxi).
\end{equation}

Step 3. We now prove that $\uh$ is also a supersolution. If not, there exists a point $x^*$ such that
\begin{equation}\label{218}
\uh(x^*)- \lambda  H_{\V}(x^*, \uh(x^*), \uh) < f(x^*).
\end{equation}

First, if $\uh(x^*)=K_M$, then from the definition of $\uh$, $x^*$ is a maximal point for $\uh$. Then $H_{\V}(x^*, \uh(x^*), \uh)\leq 0$, which contradicts with $f\leq K_M$.

Therefore, we know that $\uh(x^*)<K_M$. There exists $\epsilon>0$ such that $\uh(x^*)+\epsilon < K_M$. Define 
\begin{equation}
\vh(\vxi):= \left\{ \begin{array}{ll}
\uh(x^*)+\epsilon, & \quad \vxi=x^*;\\
\uh(\vxi), & \quad \text{otherwise}. 
\end{array} 
\right.
\end{equation}
For $\vxi=x^*$, from \eqref{218}, by the continuity of $H_{\V}$, taking $\epsilon$ small enough, we have
\begin{equation}
\uh(x^*)+\epsilon- \lambda  H_{\V}(x^*, \uh(x^*)+\epsilon, \uh) < f(x^*).
\end{equation}
For $\vxi\neq x^*$, since $\uh\leq \vh$, from \eqref{217} and the monotone condition \eqref{mono}, we have
\begin{equation}
\vh(\vxi)-  \lambda H_{\V}(\vxi, \vh(\vxi), \vh) \leq \uh(\vxi)- \lambda  H_{\V}(\vxi, \uh(\vxi), \uh) \leq  f(\vxi).
\end{equation}
Therefore, $\vh\in \mathcal{F}$ is a subsolution that is strictly larger than $\uh$. This is a contradiction. In summary, $\uh$ constructed in \eqref{def_u} is a solution. Uniqueness is directly from the construction in \eqref{def_u}. 
\end{proof}

\subsubsection{Construct barriers to control the polynomial growth of intensity $\tilde{\Phi}_j^\pm(\vxi)$}
Below, under a slight assumption of the conservation of total mass, we provide more contraction estimates on the solution to \eqref{dis_exist}. 

Assume there exists a componentwisely positive mass vector $\vm$  such that \eqref{mb_j} holds.
Under the mass balance assumption \eqref{mb_j}, for $m_1=\min_{i=1}^N m_i$, we have
\begin{equation}\label{322mm}
m_1 \|\vx\| \leq m_1 \|\vx\|_1 \leq  \vm \cdot \vx  \leq \|\vm\| \|\vx\|, \quad \vx\in\mathbb{R}^N_+,
\end{equation}
where $\|\vx\|=\sqrt{x_1^2+\cdots + x_N^2},$ and $\|\vx\|_1 = \sum_i x_i$. In fact, since $x_i>0$ and $m_i>0$, we have $\|\vx\| \leq \|\vx\|_1$ and $m_1\|\vx\|_1 \leq \vm \cdot \vx.$

\begin{prop}\label{prop:cp}
Let $f\in \ell^\8$ be RHS in the resolvent problem \eqref{dis_exist}. Assume there exists positive $\vm$ satisfying \eqref{mb_j}.
Define 
\begin{equation}
f_m(r):= \inf_{|\vxi|\geq r, \vxi\in\omp} f(\vxi),\quad f_M(r):= \sup_{|\vxi|\geq r, \vxi\in\omp} f(\vxi).
\end{equation}
Define 
\begin{equation}
u_m(\vxi):= \left\{ \begin{array}{cc}
f_m\bbs{\frac{ \vm \cdot \vxi }{\|m\|}}, \, & \vxi\in\omp,\\
f(\vxi),\, & \vxi \in \omn,
\end{array}
\right. \quad
u_M(\vxi):= \left\{ \begin{array}{cc}
f_M\bbs{\frac{ \vm \cdot \vxi }{m_1}}, \, & \vxi\in\omp,\\
f(\vxi),\, & \vxi \in \omn.
\end{array}
\right.
\end{equation}
Then
\begin{enumerate}[(i)]
\item 
 we have the estimate
\begin{equation}\label{fmM}
u_m(\vxi)  \leq f(\vxi) \leq   u_M(\vxi), \quad \forall \vxi\in \Omega_{\V};
\end{equation}
\item $u_m(\vxi)$ and $u_M(\vxi)$ are two stationary solutions to HJE
\begin{equation}\label{ssHJE}
H(\vxi, u_m(\vxi), u_m) = 0, \quad H(\vxi, u_M(\vxi), u_M) = 0;
\end{equation}
\item we have the barrier estimates for the solution $\uh$ to \eqref{dis_exist}
\begin{equation}
u_m(\vxi) \leq \uh(\vxi) \leq u_M(\vxi);
\end{equation}
\item if $f$ satisfies $f\equiv \text{const}$ for $|\vxi|>R$ and $\vxi\in\omp$, then 
   $\uh(\vxi)$ satisfies $\uh\equiv \text{const}$ for $|\vxi|> \frac{\|\vm\|R}{m_1}$ and $\vxi\in \omp$.
\end{enumerate}
\end{prop}

\begin{proof}
First, $f_{M}(r)$ is decreasing in $r$ while $f_{m}(r)$ is increasing in $r$.
Thus, together with \eqref{322mm}, we know for $\vxi\in\omp$
\begin{align}
f(\vxi) \leq f_{M}( |\vxi| ) \leq f_{M}\bbs{\frac{ \vm \cdot \vx_i }{m_1}},\\
f(\vxi) \geq f_{m}( |\vxi| ) \geq f_{m}\bbs{\frac{ \vm \cdot \vx_i }{\|\vm\|}}.
\end{align}

Second,   $\vm$ satisfies \eqref{mb_j} implies
\begin{equation}
u_m(\vxi+\vec{\nu}_j h) = f_m\bbs{\frac{\vm \cdot \vxi + \vec{\nu}_j \cdot \vm h}{\|m\|}} = u_m(\vxi), \quad \vxi\in\omp \text{ and } \vxi + \vec{\nu}_j   h \in\omp.
\end{equation}
Thus from the definition of $H_h$ in \eqref{def_Hh}, we conclude (ii).

Third,  from the monotonicity (i) in Lemma \ref{lem_nonexp_d} and \eqref{ssHJE}, we know the comparison $u_m(\vxi)\leq f(\vxi)\leq u_M(\vxi)$ in \eqref{fmM} implies $u_m(\vxi) \leq \uh(\vxi)\leq u_M(\vxi)$. Thus we have conclusion (iii) and thus (iv).
\end{proof}

\subsection{Existence and uniqueness of semigroup solution to the backward equation}\label{sec:ext_dd}

Combining the existence results and barrier estimates in Section \ref{sec:upwind}, we obtain the existence and uniqueness of the solution to the backward Euler scheme \eqref{bEuler}. This, together with the nonexpansive property of the resolvent $J_{\Delta t, h} = (I - \Delta t H_{\V})^{-1}$ in Lemma \ref{lem_nonexp_d}, implies that $-H_{\V}$ is a maximal accretive operator. Precisely, we first use the one-point Alexandroff compactification $\Omega_{\V}^*=\Omega_{\V} \cup \{\infty\}$ of $\Omega_{\V}$ \cite{Kelley} to define the Banach space
\begin{equation}
\ellb:= \{(\uh(\vxi))\in \ell^\8(\Omega_{\V});\,\, \uh(\vxi) \to \text{const as }|\vxi|\to \infty \}
\end{equation}
and a subspace
\begin{equation}
\ellc:= \{(\uh(\vxi))\in \ell^\8(\Omega_{\V});\,\, \uh(\vxi) = \text{const for } |\vxi|>R \text{ for some }R>0\} .
\end{equation}
Define
\begin{equation}\label{HHn}
H_{\V}: D(H_{\V}) \subset \ellb \to \ellb; \qquad  ( u(\vxi) )_{\vxi\in\Omega_{\V}} \mapsto ( H_{\V}(\vxi, u(\vxi), u(\cdot)) )_{\vxi\in\Omega_{\V}}
\end{equation}
The abstract domain of $H_{\V}$ is not straightforward to characterize due to the polynomial growth of $\tilde{\Phi}^{\pm}_j(\vxi)$ in $H_{\V}$. However, for $\uh\in \ellc$, the polynomial growth at far field does not turn on because $\uh=\text{const}$ for large $|\vxi|$. Hence, $\ellc\subset D(H_{\V})$. Since $\overline{\ellc} = \ellb$, we conclude that $H_{\V}$ is densely defined.
Thus, we obtain the following theorem:

\begin{thm}\label{thm:backwardEC}
Let $H_{\V}$ be the discrete Hamiltonian operator defined in \eqref{HHn}. Assume there exists a (componentwise) positive $\vm$ satisfying \eqref{mb_j}. Let $u^0\in \ellb$, then there exists a unique global solution $\uh(t,\vxi)\in C([0,\infty); \ellb)$ to \eqref{upwind}, and $\uh(t,\vxi)$ satisfies
\begin{enumerate}[(i)]
\item contraction:
\begin{equation}\label{contraction}
\inf_{\vxi\in\Omega_{\V}} u^0(\vxi) \leq \inf_{\vxi\in\Omega_{\V}} \uh(t,\vxi) \leq \sup_{\vxi\in\Omega_{\V}} \uh(t,\vxi) \leq \sup_{\vxi\in\Omega_{\V}} u^0(\vxi);
\end{equation}
\item for all $\vxi\in \Omega_{\V}$, $\uh(\cdot,\vxi)\in C^1(0,+\8);$
\item 
Moreover, if $u^0\in \ellc$   then $\uh(\cdot,\vxi)\in C^1[0,+\8)$ for all $\vxi\in \Omega_h.$
\end{enumerate}
\end{thm}

\begin{proof}
 Assume $u^0\in \ellb = \overline{ D(H_{\V}) } $. Then fix $\Delta t$, from the existence and uniqueness of the resolvent problem in
Lemma \ref{lem:perron} and the barrier estimate in Proposition \ref{prop:cp}, we know for any $f\in \ellb$, there exists a unique solution $\uh\in \ellb$ to the resolvent problem  
$$\uh(\vxi) - \Delta t  H_{\V}(\vxi, \uh(\vxi), \uh) = f(\vxi).$$
Hence we know $H_{\V}(\uh)\in \ellb$ and thus $\uh \in D(H_h)$. This gives the maximality of $-H_h$. Besides, the accretivity of $-H_h$ is given by the non-expansive property of the resolvent in Lemma \ref{lem_nonexp_d}.

   Thus  $-H_{\V}$ is a maximal accretive operator on Banach space $\ellb$, then by \textsc{Crandall and Liggett} \cite{CrandallLiggett}, there exists a unique global contraction $C^0$-semigroup solution $\uh(t,\vxi)\in C([0,+\8); \ellb)$  to \eqref{upwind}
\begin{equation}\label{semi27}
\uh(t,\cdot):=\lim_{\Delta t \to 0} (I - \Delta t H_{\V})^{-[\frac{t}{\Delta t}]} u^0=: S_{h,t}u^0.
\end{equation}
Here $[\frac{t}{\Delta t}]$ is the integer part of $\frac{t}{\Delta t}$.
   The contraction \eqref{contraction} follows from the nonexpansive property of $J_{\Delta t, h}$ in Lemma \ref{lem_nonexp_d}. 
For any $\vxi\in\Omega_{\V}$,   we have
\begin{equation}
\uh(t,\vxi) - u (0,\vxi) = \int_0^t H_{\V}(\vxi, \uh(s, \vxi), \uh(s)) \ud s.
\end{equation}
Therefore we conclude $\uh(\cdot, \vxi)\in C^1(0,+\8)$ is the solution to \eqref{upwind}. For $u^0\in \ellc$, $u^0\in D(H_{\V})$, thus the $C^1$ regularity includes $t=0.$ 
\end{proof}

\begin{rem}
We remark that since the solution to \eqref{upwind} translated by a constant is still a solution, the nonexpansive property for time-continuous monotone schemes is equivalent to the monotonicity of the solution, as shown by Crandall and Tartar's lemma \cite{Crandall_Tartar_1980}. Precisely, given an initial data $\psih^0$, denote the numerical solution computed from the monotone scheme \eqref{upwind} as
\begin{equation}
\psih(t, \cdot) = S_t \psih^0.
\end{equation}
Then we know that $S_t$ is invariant under a translation by a constant, i.e., for any constant $c$ and any function $\uh\in \ell^\8(\Omega_{\V})$,
\begin{equation}\label{invariant}
S_t(\uh + c) = S_t \uh + c.
\end{equation}
Moreover, this semigroup satisfies the following properties:
\begin{enumerate}[(i)]
\item Monotonicity: if $\uh \geq \vh$, then $S_t \uh \geq S_t \vh$;
\item $\ell^\8$ nonexpansive:
\begin{equation}
\|S_t \uh - S_t \vh\|_{\8} \leq \|\uh-\vh\|_{\8}.
\end{equation}
\end{enumerate}
Indeed, first, from the definition of $S_t$, if $\psih$ is a solution to \eqref{upwind}, then $\psih+c$ is also a solution. Thus, \eqref{invariant} holds. Second, similar to \eqref{mp1}, we have (ii). Then, by Crandall and Tartar's lemma \cite{Crandall_Tartar_1980}, under \eqref{invariant}, (i) is equivalent to (ii).
\end{rem}

\subsection{Existence, comparison principle and exponential tightness of CME for reversible case}\label{sec:db_d}

In this section, under the assumption of a positive reversible invariant measure $\pi_{\V}$ satisfying \eqref{master_db_n}, we focus on recovering the existence and exponential tightness of the CME from the backward equation.

\subsubsection{CME Reversibility Condition}

We first review the CME reversibility condition. Since the jumping process with $Q_{\V}$ only distinguishes the same reaction vector $\vec{\xi}$, we rearrange all $j$ such that $\vec{\nu}_j=\pm\vec{\xi}$ and define the grouped fluxes as
\begin{equation}\label{group_flux}
\tilde{\Phi}^+_{\xi}(\vx) := \sum_{j: \vec{\nu}_j = \vec{\xi}} \tilde{\Phi}^+_j(\vx) + \sum_{j: \vec{\nu}_j = -\vec{\xi}} \tilde{\Phi}^-_j(\vx), \quad \tilde{\Phi}^-_{\xi}(\vx) := \sum_{j: \vec{\nu}_j = \vec{\xi}} \tilde{\Phi}^-_j(\vx) + \sum_{j: \vec{\nu}_j = -\vec{\xi}} \tilde{\Phi}^+_j(\vx).
\end{equation}
Then the CME \eqref{rp_eq} can be recast as
\begin{equation}\label{rp_eq_nn}
\begin{aligned}
\frac{\partial}{\partial t} p_{\V}(\vxi, t) = &\frac{1}{h}\sum_{\vec{\xi}, \vxi-\vec{\xi} h\geq 0} \left(\tilde{\Phi}^+_\xi(\vxi-\vec{\xi} h) p_{\V}(\vxi-\vec{\xi} h, t) - \tilde{\Phi}^-_\xi(\vxi) p_{\V}(\vxi, t)\right) \\
&+ \frac{1}{h}\sum_{\vec{\xi}, \vxi+h\vec{\xi} \geq 0} \left(\tilde{\Phi}^-_\xi(\vxi+\vec{\xi} h) p_{\V}(\vxi+\vec{\xi} h, t) - \tilde{\Phi}_\xi^+(\vxi) p_{\V}(\vxi, t)\right), \quad \vxi\in\omp.
\end{aligned}
\end{equation}
Again, $\frac{\partial}{\partial t} p_{\V}(\vxi, t) = 0$ for $\vxi\in\omn$. Therefore, for any $\vxi$ and $\vec{\xi}$, the reversibility condition for the "effective stochastic process" means that the total forward probability steady flux from $\vxi$ to $\vxi+\vec{\xi}h$ equals the total backward one. In other words, for $\vxi\in\omp$ and $\vxi+ \vec{\xi}_{j}h\in\omp$,
\begin{equation}\label{master_db_n}
\begin{aligned}
\tilde{\Phi}^-_{\xi}(\vxi+ \vec{\xi}_{j}h)     \pi_{\V}(\vxi+ {\vec{\xi}_{j}}h) = \tilde{\Phi}^+_{\xi}(\vxi)\pi_{\V}(\vxi), \quad \forall \vec{\xi}.
\end{aligned}
\end{equation}
We call this as a reversible condition   for CME; a.k.a Markov chain detailed balance, c.f. \cite{Joshi}.    It is well known that under a chemical version of constrained detail balance condition, the invariant measure $\pi_{\V}$ is given by the product Poisson and  the limit of $\psihs(\vxi)=- h \log \pi_{\V}(\vx_{\V})$ gives a convex viscosity solution to stationary HJE; see  for example, \cite[Ch7, Thm 3.1 and eq. (1) in Sec 7.4]{Whittle}, \cite[Theorem 4.5]{Anderson_Craciun_Kurtz_2010}, \cite[(7.30)]{QianBook}, \cite[Section 3]{GL22}.

Under the reversible condition for CME \eqref{master_db_n}, due to the flux grouping degeneracy brought by the same reaction vector, we no longer have an explicit invariant distribution. In fact, in some non-equilibrium enzyme reactions that satisfy the reversible condition \eqref{master_db_n}, the macroscopic energy landscape $\psi^{ss}$ becomes nonconvex with multiple minima \cite[Section 4]{GL22}. These multiple stable states are known as non-equilibrium steady states (NESS), which were pioneered by Prigogine \cite{prigogine1967introduction}. Multiple steady states, nonzero steady state fluxes, and positive entropy production rates at NESS are three distinguished features of non-equilibrium reactions. 

In the following, we provide a rigorous proof for the existence and uniqueness of the limiting energy landscape $\psi^{ss}$ under the reversible condition for CME \eqref{master_db_n}.

\subsubsection{Existence and Comparison Principle for Reversible CME}

In this section, under the assumption that the CME satisfies the reversible condition \eqref{master_db_n}, we prove the existence and comparison principle for the solution to the CME \eqref{rp_eq}.

\begin{lem}
Assume there exists a positive invariant measure $\pi_{\V}>0$ in $\omp$ for the CME \eqref{rp_eq}, and $\pi_{\V}$ satisfies the reversible condition \eqref{master_db_n}. Then, we have
\begin{equation}
Q_{\V}^* p_{\V} = \left\{ 
\begin{array}{cc}
\pi_{\V} Q_{\V} \left(\frac{p_{\V}}{\pi_{\V}}\right), & \vxi\in\omp, \\
0, & \vxi\in \omn.
\end{array} 
\right.
\end{equation}
\end{lem}

\begin{proof}
Using \eqref{master_db_n},    the right-hand side of CME \eqref{rp_eq} reads that for $\vxi\in\omp$, 
\begin{equation} 
\begin{aligned}
\frac{1}{\pi_{\V}}Q^*_{\V} p_{\V} = &\frac{1}{h}\sum_{\vec{\xi}, \vxi- \vec{\xi} h\geq 0}   \tilde{\Phi}^+_\xi(\vxi-\vec{\xi} h )     \frac{p_{\V}(\vxi- \vec{\xi} h,t)}{\pi_{\V}(\vxi)} -      \tilde{\Phi}^-_\xi(\vxi)     \frac{p_{\V}(\vxi,t)}{\pi_{\V}(\vxi)} \\
 &  +  \frac1h\sum_{\vec{\xi}, \vxi+h\vec{\xi}  \geq 0}   \tilde{\Phi}^-_\xi(\vxi+ \vec{\xi} h )     \frac{p_{\V}(\vxi+ \vec{\xi} h,t)}{\pi_{\V}(\vxi)} -  \tilde{\Phi}_\xi^+(\vxi) \frac{p_{\V}(\vxi,t)}{\pi_{\V}(\vxi)}\\
 =& \frac{1}{h}\sum_{\vec{\xi}, \vxi- \vec{\xi} h\geq 0}   \tilde{\Phi}^-_\xi(\vxi)    \bbs{ \frac{p_{\V}(\vxi- \vec{\xi} h,t)}{\pi_{\V}(\vxi- \vec{\xi} h)} -        \frac{p_{\V}(\vxi,t)}{\pi_{\V}(\vxi)}} \\
 &  +  \frac1h\sum_{\vec{\xi}, \vxi+h\vec{\xi}  \geq 0}   \tilde{\Phi}^+_\xi(\vxi)\bbs{     \frac{p_{\V}(\vxi+ \vec{\xi} h,t)}{\pi_{\V}(\vxi+ \vec{\xi} h)} -    \frac{p_{\V}(\vxi,t)}{\pi_{\V}(\vxi)}} = Q_{\V} \bbs{\frac{p_{\V}}{\pi_{\V}}}.
 \end{aligned}
\end{equation}
\end{proof}
\begin{defn}
Given any $0<\ell<\8$, if there exists a $R_\ell$ and $h_0$ such that, 
\begin{equation}
\sup_{t\in[0,T]}\sum_{|\vxi|\geq R_{\ell}, \vxi\in\omp} p_{\V}(\vxi,t) \leq e^{-\frac{\ell}{h}}, \quad \forall h\leq h_0,
\end{equation}
then we say the sequence of process $(\cv(t))$ is exponentially tight for each  $t\in[0,T].$ 
\end{defn}

Here we only consider $\vxi\in\omp$ because from \eqref{rp_eq}, it is natural to take $p_{\V}(\vxi, 0) = 0$ for $\vxi\in\omn$, which remains to be $0$ all the time for $\vxi\in\omn.$
Notice that
\begin{equation}
\sup_{t\in[0,T]} \bP\bbs{|\cv(t)|\geq R_{\ell}} \leq \bP\bbs{\sup_{t\in[0,T]}|\cv(t)|\geq R_{\ell}}.
\end{equation}
Hence the exponential tightness of $(\cv(t))$ for each $t\in[0,T]$ defined above is weaker then the usual definition of the exponential compact containment
condition for $\cv$, i.e.,
for any $\ell<\8$,  there exist  $R_\ell$ and $h_0$ such that, 
\begin{equation}
\bP(\sup_{t\in[0,T]}|\cv(t)|\geq R_{\ell}) \leq e^{-\frac{\ell}{h}}, \quad \forall h\leq h_0;
\end{equation}
 c.f. \cite[Theorem 4.4]{feng2006large}.

Then using Theorem \ref{thm:backwardEC}, we obtain  the existence and comparison principle for  $\uh = h\log w_{\V}$ and thus the solution $w_{\V}$ to backward equation \eqref{backward}. We remark that  one can also directly prove the existence and comparison principle for \eqref{backward} using linear semigroup theory developed by \textsc{Lumer–Phillips} in 1961. Therefore, under the detailed balance assumption of the invariant measure $\pi_{\V}$,  we have the following corollary.
\begin{cor}\label{cor:tight}
Assume there exists a positive reversible invariant measure $\pi_{\V}>0$ in $\omp$ to CME \eqref{rp_eq} such that \eqref{master_db_n} holds. We have
\begin{enumerate}[(i)]
\item there exists a unique global solution $p_{\V}(\vxi,t)=\pi_{\V}(\vxi)\uh(\vxi,t)$ in $\omp$, where $\uh$ is the solution to \eqref{upwind} with initial data $\frac{p_{{\V}}^0}{\pi_{\V}}$ in $\omp$;
\item if the initial density satisfies $c_1 \pi_{\V} \leq p^0_{\V} \leq c_2 \pi_{\V}$ in $\omp$, then
\begin{equation}
c_1 \pi_{\V} \leq p_{\V} \leq c_2 \pi_{\V} \quad \text{ in } \omp;
\end{equation}
\item if $\pi_{\V}$ is exponentially tight, then $p_{\V}$ is also exponentially tight for any $t$.
\end{enumerate}
\end{cor}

Now we state a lemma which also ensures the exponential tightness as long as the initial density is compact support.
\begin{lem}\label{tight2}
Assume there exists (componently) positive $\vm$ satisfying \eqref{mb_j} and the 
  initial distribution $p^0_{\V}$ has a compact support. Then $\cv$ is   exponentially tight for any $t$.
\end{lem}
\begin{proof}
Under the assumption (ii), then \eqref{322mm} holds.
Using the mass balance vector \eqref{mb_j}, multiplying SDE \eqref{Csde} by $\vm$, we have
\begin{equation}
\cv(t) \cdot \vm = \cv(0) \cdot \vm.
\end{equation}
For  $\cv(0) \in \omp$, we know $\cv(t) \in \omp$ and obtain 
\begin{equation}
m_1 \|\cv(t)\| \leq \|m\|\|\cv(0)\|, \quad \forall t\geq 0
\end{equation}
and thus the associated distribution $\bP$ of $\cv$ has a compact support for any time $t$. 
We obtain the exponential tightness of $\cv(t)$ at each time $t$.
\end{proof}

\begin{rem}\label{rem38}
For chemical reaction detailed balance case \eqref{DB}, we know $\psi_{\V} = - h \log \pi_{\V} \to \KL(\vx||\peq)$. Since $\KL(\vx||\peq)$ is super linear, so we know for $|\vxi|\geq R\gg 1$, $\vxi\in\omp$
\begin{equation}
\psi_{\V} \geq \frac{1}{2} \KL(\vx||\peq) \geq c|\vx|.
\end{equation}
Thus
\begin{equation}
\sum_{|\vxi| \geq R} e^{-\frac{\psi_{\V}}{h}} \leq \int_{|\vx|\geq R} e^{-\frac{c|\vx|}{h}} \leq e^{-\frac{c R}{2h}}.
\end{equation}
For any $\ell<\8$, taking $R=\frac{2\ell}{c}$, we obtain $\pi_{\V}$ has exponential tightness and hence by Corollary \ref{cor:tight}, the chemical process satisfies the exponential tightness for any $t$.
\end{rem}

\section{Thermodynamic limit  of CME and backward equation in the HJE forms}\label{sec4}

The comparison principle and nonexpansive properties for the monotone schemes naturally bring up the concept of viscosity solutions to the corresponding continuous HJE as $h\to 0$; see Appendix \ref{app3} for the definition of viscosity solutions. Indeed, the monotone scheme approximation and the vanishing viscosity method are two ways to construct a viscosity solution to HJE proposed by Crandall and Lions in \cite{crandall1983viscosity, crandall1984two}.

In this section, we first study the existence of a upper semicontinuous (USC) viscosity solution to the stationary HJE in the Barron-Jensen sense; see Proposition \ref{prop:USC}. This relies on the assumption of the existence of a positive reversible invariant measure $\pi_{\V}$ for CME. We point out that our reversible assumption \eqref{master_db_n} is slightly more general than the commonly used chemical reaction detailed balance condition. Therefore, our invariant measure includes some non-equilibrium enzyme reactions, and the stationary HJE solution can be nonconvex, indicating a nonconvex energy landscape for non-equilibrium reactions \cite{GL22}.

Second, we focus on the viscosity solution to the dynamic HJE \eqref{HJE2psi}. Fixing $\Delta t$ in the backward Euler scheme \eqref{bEuler} and taking the limit $h\to 0$, following Barles and Perthame's procedure, the USC envelope of the discrete HJE solution gives a USC subsolution to HJE, which automatically inherits the 'no reaction' boundary condition from \eqref{bEuler}. The proof still relies on the monotonicity and nonexpansive property of the monotone schemes. Thanks to the conservation of mass, we use some mass functions $f(\vm \cdot \vx)$ to construct barriers to control the polynomial growth of coefficients in the Hamiltonian. Then, by the comparison principle, we obtain the viscosity solution to the continuous resolvent problem. Next, taking $\Delta t \to 0$, thanks to Crandall and Liggett's construction of a nonlinear contraction semigroup solution \cite{CrandallLiggett}, we finally obtain a unique viscosity solution to the HJE as a large size limit of the backward equation; see Theorem \ref{thm_vis}.

\subsection{Stationary HJE and convergence from the mesoscopic reversible invariant measure}\label{sec:sHJE}

The stationary solution to the HJE, $H(\nabla \psi^{ss}(\vx), \vx) = 0$, can be used to compute the energy landscape guiding the chemical reaction, to decompose the RRE to dissipative and conservative parts and to compute the associated energy barrier of transition paths; see  \cite{GL22}.    However, $\psi^{ss}$ is usually challenging to compute and is non-unique. Indeed, the structure of chemical reaction network can only be revealed through the global energy landscape obtained from the large deviation rate function for the invariant measure. In this section, we assume the existence of an invariant measure for the mesoscopic CME \eqref{rp_eq} satisfying the reversible condition \eqref{master_db_n} and then use it to obtain convergence to the stationary solution of the HJE.

Under the assumption of the existence of a positive reversible invariant measure, we now prove that this invariant measure converges to a Barron-Jensen's  USC viscosity solution.

As in \eqref{1.3}, we change the variable in the CME to $\psihs(\vxi)$ such that
$$\pi_{\V}(\vxi) = e^{-\psihs(\vxi)/h}.$$
Then the reversible condition for CME \eqref{master_db_n} becomes, for $\vxi\in\omp$ and $\vxi+ \vec{\xi}_{j}h\in\omp$,
\begin{equation}\label{db_psi}
\tilde{\Phi}^-_{\xi}(\vxi+ \vec{\xi}_{j}h)   e^{\frac{\psihs(\vxi) - \psihs(\vxi+\vec{\nu}_j h)}{h}} = \tilde{\Phi}^+_{\xi}(\vxi), \quad \forall \vec{\xi}.
\end{equation}
The solvability follows from the assumption that $\pi$ exists and satisfies the reversible condition for CME \eqref{master_db_n}.

Let $\ph(\vxi), \, \vxi\in \omp$, be the solution to \eqref{db_psi}. Take some continuous function $\psi(\vx)$ in the domain $\bR^N\backslash \bR^N_+$. Define the extension of $\ph$ in $\omn$ as 
\begin{equation}\label{exC}
\ph(\vxi)=\psi(\vxi), \quad \vxi\in\omn.
\end{equation}
Denote its upper semicontinuous (USC) envelope as
\begin{equation}\label{USC}
\bar{\psi}(\vx):=\limsup_{h\to 0^+, \vxi \to \vx}  \ph(\vxi).
\end{equation}
Similarly, we use the lower semicontinuous (LSC) envelope to construct a supersolution
\begin{equation}\label{LSC}
\underline{\psi}(\vx):=\liminf_{h\to 0^+, \vxi \to \vx}  \ph(\vxi).
\end{equation}
The next proposition states that, after taking the large size limit $h\to 0$ from a reversible invariant measure, we can obtain a USC viscosity solution to the stationary HJE in the sense of \textsc{Barron-Jensen} \cite[Definition 3.1]{Barron_Jensen_1990}.
\begin{prop}\label{prop:USC}
Let $\ph(\vxi)$ be a solution to \eqref{db_psi} with a continuous extension \eqref{exC}. Then, as $h\to 0$, the upper semicontinuous envelope $\bar{\psi}(\vx)$ is a USC viscosity solution to the stationary HJE in the Barron-Jensen's sense. That is, for any smooth test function $\varphi$, if $\bar{\psi}-\varphi$ has a local maximum at $\vx_0$, then
\begin{equation}
H(\nabla \varphi(\vx_0), \vx_0)=0.
\end{equation}
\end{prop}

\begin{proof}
Step 1. For any smooth test function $\varphi$, let $x^*$ be a strict local maximal of $\bar{\psi}-\varphi$. Denote $c_0:=\max_j |\vec{\nu}_j|$. Then for some $r,$ there exists $c>0$ such that
\begin{equation}
\bar{\psi}(x_0) - \varphi(x_0) \geq \bar{\psi}(x) - \varphi(x) + c|x-x_0|^2, \quad x\in B(x_0, r).
\end{equation}
Then taking $c$ large enough, as proved by \textsc{Barles, Souganidis} in \cite[Theorem 2.1]{barles1991convergence}, there exists a sequence $\{\ph\}$ with $ x_{*}^{h}$ being the maximum point of $\ph -\varphi$ in $B(x_*^h, c_0 h)$ and satisfies
\begin{equation}
 x_{*}^{h} \to x_0, \quad \ph( x_{*}^{h}) \to \bar{\psi}(x_0).
\end{equation} 

Step 2. Since $\ph$ is the discrete solution to \eqref{db_psi}, 
\begin{equation}\label{tm_db_1}
\tilde{\Phi}^-_{\xi}(\vxi+ \vec{\xi}_{j}h)   e^{\frac{\ph(\vxi) - \ph(\vxi+\vec{\xi}_j h)}{h}} = \tilde{\Phi}^+_{\xi}(\vxi), \quad \forall \vec{\xi}.
\end{equation} 
Since for $x\in B(x_*^h, c_0 h)$, we have
\begin{equation}\label{tm_test1}
\ph(x_{*}^{h}) - \varphi(x_{*}^{h}) \geq \ph(x) - \varphi(x), 
\end{equation}
thus $\ph(x_{*}^{h}) -  \ph(x) \geq \varphi(x_{*}^{h}) - \varphi(x).$
 Thus replacing $\ph$ by $\varphi$ in \eqref{tm_db_1} yields
\begin{equation}\label{l1}
\tilde{\Phi}^-_{\xi}(x_*^h+ \vec{\xi}_{j}h)   e^{\frac{\varphi(x_*^h) - \varphi(x_*^h+\vec{\xi}_j h)}{h}} \leq  \tilde{\Phi}^+_{\xi}(x_*^h), \quad \forall \vec{\xi}.
\end{equation}

On the other hand, $\ph$ also satisfies
\begin{equation}\label{tm_db_2}
\tilde{\Phi}^+_{\xi}(\vxi- \vec{\xi}_{j}h)   e^{\frac{\ph(\vxi) - \ph(\vxi-\vec{\xi}_j h)}{h}} = \tilde{\Phi}^-_{\xi}(\vxi), \quad \forall \vec{\xi}
\end{equation} 
due to \eqref{db_psi}.
Hence same as \eqref{tm_test1}, using  $\ph(x_{*}^{h}) -  \ph(x) \geq \varphi(x_{*}^{h}) - \varphi(x)$, replacing $\ph$ by $\varphi$ in \eqref{tm_db_2} yields 
\begin{equation}\label{l2}
\tilde{\Phi}^+_{\xi}(x_*^h- \vec{\xi}_{j}h)   e^{\frac{\varphi(x_*^h) - \varphi(x_*^h-\vec{\xi}_j h)}{h}} \leq  \tilde{\Phi}^-_{\xi}(x_*^h), \quad \forall \vec{\xi}.
\end{equation}
Then talking limit $h\to 0$ in \eqref{l1} and \eqref{l2} implies
\begin{equation}
\tilde{\Phi}^-_{\xi}(x_0)   e^{-\nabla\varphi(x_0)} =  \tilde{\Phi}^+_{\xi}(x_0).
\end{equation}
\end{proof}

\begin{rem}
This notion of USC viscosity solution was first proposed by \textsc{Barron and Jensen} \cite{Barron_Jensen_1990}. They proved the uniqueness of the solution in this new notion for evolutionary problems. However, the uniqueness of the solution in this new notion for the stationary problem was left as an open problem.
On the other hand, if this USC viscosity solution can be proved to be continuous, then by \cite[Theorem 2.3]{Barron_Jensen_1990}, this USC viscosity solution is indeed the viscosity solution in the sense of \textsc{Crandall, Lions} viscosity solution \cite{crandall1983viscosity}.
\end{rem}

\begin{rem}
We point out that the assumption of the reversibility of $\pi_{\V}$ includes both the chemical reaction detailed balance case 
and some non-equilibrium enzyme reactions \cite{GL22}, where the energy landscape is proved to be the USC viscosity solution to the stationary HJE in the Barron-Jensen's sense (see Proposition \ref{prop:USC}). The resulting energy landscape is nonconvex with multiple steady states, indicating three key features of non-equilibrium reactions: multiple steady states, non-zero flux, and a positive entropy production rate at the non-equilibrium steady states (NESS), as advocated by \textsc{Prigogine} \cite{prigogine1967introduction}.
\end{rem}

\begin{rem}
In the chemical reaction detailed balance case \eqref{DB}, one has a closed formula for the unique invariant measure $\pi_{\V}$ given by the product Poisson distribution \eqref{pos}. However, for the corresponding macroscopic stationary HJE $H(\nabla \psi^{ss}(\vx),\vx)=0$, we still do not have the uniqueness of the viscosity solution since a constant is always a solution. This means a selection principle such as the weak KAM solution needs to be used to identify a unique solution; see \cite{GL23} for the drift-diffusion case. We believe this selection principle shall be consistent with the physically meaningful stationary solution, which is constructed by the limit of the invariant measure $\pi_{\V}$ for the CME as shown in Proposition \ref{prop:USC}. We leave this deep question for future study.
\end{rem}

\subsection{Convergence of backward equation to viscosity solution of HJE}\label{sec:vis}

Define the Banach space
\begin{equation}\label{space}
\buc:= \{u\in C(\mathbb{R}^N);\,\,  u(\vx)  \to \text{const} \text{ as } |\vx|\to +\infty\},
\end{equation}
where $\mathbb{R}^{N*}=\mathbb{R}^N \cup \{\infty\}$ is the one-point Alexandroff compactification of $\mathbb{R}^N$ \cite{Kelley}.
In this section, we will follow the framework of \textsc{Crandall, Lions} \cite{crandall1983viscosity} and \textsc{Feng, Kurtz} \cite{feng2006large} to first prove that the Hamiltonian, regarded as an operator on the Banach space $\buc$, is a $m$-accretive operator and then obtain the dynamic solution of HJE as a nonlinear semigroup contraction generated by $H$.

To overcome the polynomial growth of the coefficients in the Hamiltonian $H$, we further define a subspace of $\buc$ as
\begin{equation}
\ccc:= \{u\in \buc; \,\, u\equiv\text{const} \text{ for } |\vx|\geq R \text{ for some } R\}.
\end{equation}
Recalling the Hamiltonian $H(\nabla u(\vx), \vx)$ in \eqref{H}, we define the operator
\begin{equation}
H: D(H)\subset \buc \to \buc \quad \text{such that } u(\vx) \mapsto H(\nabla u(\vx), \vx).
\end{equation}
The maximality of $H$ such that $\text{ran}(I-\Delta t H)=\buc$ requires finding a classical solution to the HJE, which is generally difficult. This is because (i) the solution to the HJE does not have enough regularity; (ii) due to the polynomial growth of the coefficients in the Hamiltonian, we cannot find a solution in commonly used spaces, such as $C_b(\mathbb{R}^N)$.

Instead, we will use the notion of viscosity solution to extend the domain of $H$ so that the maximality of $H$ in the sense of viscosity solution is easy to obtain as the limit of the backward Euler approximation when $h\to 0$. We refer to Appendix \ref{app3} for the definition of viscosity solutions. Additionally, we need to construct a $\buc$ solution as a limit of a $\ccc$ viscosity solution obtained above.

Precisely, in the first step, following \cite{crandall1983viscosity} and \cite{feng2006large}, we define the viscosity extension of $H$ as
\begin{equation}\label{operatorH}
\hat{H}_1: D(\hat{H}_1)\subset \ccc \to \ccc
\end{equation}
satisfying
\begin{enumerate}[(i)]
\item $u\in D(\hat{H}_1)\subset \ccc$ if and only if there exists $f\in \ccc$ such that the resolvent problem
$$u(\vx)-\Delta t H(\nabla u(\vx), \vx) = f(\vx)$$
has a unique viscosity solution $u\in \ccc$.
\item For $u\in D(\hat{H}_1)$, $\hat{H}_1(u):= \frac{u-f}{\Delta t}$ for such an $f(\vx)$ in (i).
\end{enumerate}
In other words, the domain $D(\hat{H}_1)$ is defined as the range of the resolvent operator
\begin{equation}\label{con_J}
J_{\Delta t}:= (I-\Delta t \hat{H}_1)^{-1}.
\end{equation}
Since the resolvent problem holding in the classical sense implies that the resolvent problem has a viscosity solution, we can regard $\hat{H}_1$ as an extension of $H$. We will prove later in Proposition \ref{prop_h_limit} the existence and comparison principle for the resolvent problem in $\ccc$ above.
We remark that the domain $D(\hat{H}_1)$ characterized in (i) does not depend on the value of $\Delta t<\infty$. Indeed, if $u\in D(\hat{H}_1)$ such that $u_1-H(\nabla u_1(x), x)=f_1$ has a viscosity solution $u_1\in \ccc$ for some $f_1\in \ccc$, then $u_2-\Delta t H(\nabla u_2(x),x) = f_2$ also has a viscosity solution with $f_2= \Delta t f_1+ (1-\Delta t)u_1\in \ccc$.

Next, we extend $\ccc$ to the Banach space $\buc$. Notice that $\ccc$ is dense in $\buc$. For any $f\in \buc$, there exists a Cauchy sequence $f_k\in \ccc$ such that $\|f_k - f\|_{\infty} \to 0$. Then, from the comparison principle of the resolvent problem in $\ccc$ \cite[Page 293, Theorem 5.1.3]{bardi1997optimal}, we know that the corresponding viscosity solutions $u_k\in \ccc$ of the resolvent problem with $f_k$ form a Cauchy sequence in $\buc$. Let $u$ be the limit of $u_k$ in $\buc$. This limit is independent of the choice of the Cauchy sequence. Thus, we further define the viscosity extension of $H$ as
\begin{equation}\label{operatorH_n}
\hat{H}: D(\hat{H})\subset \buc \to \buc
\end{equation}
satisfying
\begin{enumerate}[(i)]
\item $u\in D(\hat{H})\subset \buc$ if and only if there exists $f\in \buc$ such that the resolvent problem
$$u(\vx)-\Delta t H(\nabla u(\vx), \vx) = f(\vx)$$
has a unique viscosity solution $u\in \buc$.
\item For $u\in D(\hat{H})$, $\hat{H}(u):= \frac{u-f}{\Delta t}$ for such an $f(\vx)$ in (i).
\end{enumerate}

Below, we show that the abstract domain $D(\hat{H})$ above indeed includes a large subspace.

\begin{lem}\label{lem_inc}
We have the following inclusions:
\begin{align}
C^1_{c}(\mathbb{R}^{N*}) \subset D(H) \subset D(\hat{H}_1) \subset D(\hat{H});\label{inc1}\\
\buc = \overline{D(H)}^{\|\cdot\|_{\infty}} = \overline{D(\hat{H}_1)}^{\|\cdot\|_{\infty}} = \overline{D(\hat{H})}^{\|\cdot\|_{\infty}}\label{inc2}.
\end{align}
\end{lem}

\begin{proof}
First, based on the definition of $D(\hat{H})$, it is obvious that $D(H)\subset D(\hat{H})$, because for $u$ such that $H(\nabla u(\vx),\vx)\in \buc$ holds in the classical sense, it implies that there exists $f\in \buc$ such that $(I-\Delta t H)u = f$ holds in the viscosity solution sense. Also, $D(\hat{H}_1)\subset D(\hat{H})$ is obvious because if there exists $f\in \ccc \subset \buc$, then we have a viscosity solution $u\in \ccc\subset \buc$.

Second, for $u\in C^1_{c}(\mathbb{R}^{N*})$, $\nabla u=0$ outside a ball $B_r$. Thus, $H(\nabla u(\vx), \vx)$ is bounded for $\vx\in B_r$, while $H(\nabla u(\vx),\vx)=0$ outside $B_r$. Hence, $H(\nabla u(\vx),\vx)\in \ccc$ and $u\in D(H)$. So we conclude \eqref{inc1}.

Third, since $\ccc\subset \buc$ and $\buc$ is a closed space, we have $\buc = \overline{C^1_{c}(\mathbb{R}^{N*}) }^{\|\cdot\|_{\infty}} \subset \overline{D(\hat{H})}^{\|\cdot\|_{\infty}} \subset \buc.$
\end{proof}

In the next subsection, we construct a viscosity solution to the resolvent problem
\begin{equation}\label{con_J}
u = J_{\Delta t} f \quad \text{with} \quad J_{\Delta t}= (I-\Delta t \hat{H})^{-1}.
\end{equation}
This solvability in the viscosity solution sense gives the maximality of $\hat{H}$.

Then the nonexpansive property of $J_{\Delta t}$ can be shown after taking limit from the nonexpansive property of the discrete resolvent $J_{\Delta t, h}$ proved in Lemma \ref{lem_nonexp_d}.
\subsubsection{Barles-Perthame's procedure of  convergence to viscosity solution as $h\to 0$}
In this section, we first fix $\Delta t$ and take $h\to 0$ to construct a viscosity solution to
the backward Euler problem
\begin{equation}\label{resolventP}
(I-\Delta t \hat{H}) u^n(\vx) = u^{n-1}(\vx), \quad \vx\in \bR^N.
\end{equation}
For easy presentation, this reduces to solving
 the following resolvent equation
\begin{equation}\label{resolventP_n}
(I-\Delta t \hat{H}) u(\vx) = f(\vx), \quad \vx\in \bR^N.
\end{equation}
The following proposition follows Barles-Perthame's procedure  \cite{Barles_Perthame_1987} to use the upper semicontinuous (USC) envelope of the  numerical approximation to construct a subsolution to \eqref{resolventP}. Let $u_{\V}(\vxi), \, \vxi\in \Omega_{\V}$ be the solution to \eqref{dis_exist}, and define 
\begin{equation}\label{USC}
\bar{u}(\vx):=\limsup_{h\to 0^+, \vxi \to \vx}  u_{\V}(\vxi).
\end{equation}
 Similarly, we use the lower semicontinuous (LSC) envelope to construct a supersolution
 \begin{equation}\label{LSC}
\underline{u}(\vx):=\liminf_{h\to 0^+, \vxi \to \vx}  u_{\V}(\vxi).
 \end{equation}
 We denote the set of upper semicontinuous functions on $\bR^N$ as $\usc.$

 Below, we impose the local Lipschitz continuity condition for ${\Phi}^\pm_j(\vx), \, \vx\in \mathbb{R}^N$ after zero extension. This condition is to guarantee the comparison principle for the viscosity sub/super solution constructed in the following proposition. We remark that this local Lipschitz continuity condition can be weakened, but the comparison principle requires more estimates; see \cite{Deng_Feng_Liu_2011, Kraaij_Mahe_2020}. Similar assumptions on the vanishing rate of fluxes near the boundary are studied for the sample path large deviation principle \cite{Agazzi_Andreis_Patterson_Renger_2021}.
 \begin{prop}\label{prop_h_limit}
Let $u_{\V}(\vxi)$ be a solution to \eqref{dis_exist} with $f\in \ell^{\8}_{c*}$. Assume ${\Phi}^\pm_j(\vx), \, \vx\in \bR^N$, after zero extension, is local Lipschitz continuous.  Assume there exists (componently) positive $\vm$ satisfying \eqref{mb_j}. Then as $h\to 0$, the upper semicontinuous envelope $\bar{u}(\vx)$ is a subsolution to \eqref{resolventP}. The lower semicontinuous envelope $\underline{u}(\vx)$ is a supersolution to \eqref{resolventP}. Furthermore, from the comparison principle, 
\begin{equation}\label{dis_h_limit}
u(\vx) = \bar{u} = \underline{u} = \lim_{h\to 0, \vxi \to \vx} u_{\V}(\vxi)\in \ccc.
\end{equation}
\end{prop}

\begin{proof}
Step 1. From the barrier estimates in  Proposition \ref{prop:cp}, we know $\uh(\vxi)\equiv  c_0$ for $|\vxi|> R$ for some $R>0$ and for some constant $c_0$. Thus 
\begin{equation}
\bar{u}(\vx)=\underline{u}(\vx)=c_0 \,\, \text{ for } |\vxi|>R.
\end{equation}

Step 2. For any test function $\varphi$, let $x_0$ be a strict local maximal of $\bar{u}-\varphi$. Denote $c_0:=\max_j |\vec{\nu}_j|$. Then for some $r,$ there exists $c>0$ such that
\begin{equation}
\bar{u}(x_0) - \varphi(x_0) \geq \bar{u}(x) - \varphi(x) + c|x-x_0|^2, \quad x\in B(x_0, r).
\end{equation}
Then taking $c$ large enough, as proved by \textsc{Barles, Souganidis} in \cite[Theorem 2.1]{barles1991convergence}, there exists a sequence $\{u_{\V}\}$ with $ x_{*}^{h}$ being the maximum point of $u_{\V} -\varphi$ in $B(x_*^h, c_0 h)$ and satisfies
\begin{equation}
 x_{*}^{h} \to x_0, \quad u_{\V}( x_{*}^{h}) \to \bar{u}(x_0).
\end{equation} 

Step 3. Since $u_{\V}$ is the discrete solution to \eqref{dis_exist}, 
\begin{equation} \label{tm_p1}
 u_{\V}(x_{*}^{h}) - \lambda H_{\V}(x_{*}^{h}, u_{\V}(x_{*}^{h}), u_{\V}) = f(x_{*}^{h}).
\end{equation} 
Since for $x\in B(x_*^h, c_0 h)$, we have
\begin{equation}
u_{\V}(x_{*}^{h}) - \varphi(x_{*}^{h}) \geq u_{\V}(x) - \varphi(x), 
\end{equation}
so  $u_{\V}(x)- u_{\V}(x_{*}^{h})    \leq \varphi(x)- \varphi(x_{*}^{h}) .$
 Thus replacing $u_{\V}$ by $\varphi$ in \eqref{tm_p1},  we have
\begin{equation}
H_{\V}(x_{*}^{h}, u_{\V}(x_{*}^{h}), u_{\V}) \leq H_{\V}(x_{*}^{h}, \varphi(x_{*}^{h}), \varphi).
\end{equation}
Thus taking limit $h\to 0$,
\begin{equation} 
  \bar{u}(x_0) - \lambda H(x_0, \nabla\varphi(x_0)) \leq  f(x_0).
\end{equation} 

Step 4, similarly, we can prove the LSC envelope $\underline{u}$ is a supersolution. Then by the comparison principle of \eqref{resolventP} in a ball $B_R$ \cite[Page 293, Theorem 5.1.3]{bardi1997optimal}, we have $\bar{u}\leq \underline{u}$ and thus
\begin{equation}
u(\vx) = \lim_{h\to 0, \vxi \to \vx} u_{\V}(\vxi)
\end{equation}
is the unique viscosity solution to \eqref{resolventP}.
\end{proof}

\begin{cor}\label{cor:semi}
Assume ${\Phi}^\pm_j(\vx), \, \vx\in \bR^N$, after zero extension, is local Lipschitz continuous. Assume there exists (componentwisely) positive $\vm$ satisfying \eqref{mb_j}.
\begin{enumerate}[(i)]
\item Let $f\in \ccc$, then the limit solution $u(\vx)\in \ccc$ obtained in Proposition \ref{prop_h_limit} is the unique viscosity solution to \eqref{resolventP}.
\item Let $u_1$ and $u_2$ be viscosity solutions to \eqref{resolventP}  in (i) corresponding to $f_1$ and $f_2$.  Then $\|u_1- u_2\|_{\8}\leq \|f_1 - f_2\|_{ \8}$.
\item Let $f\in \buc$, then  there exists a unique viscosity solution to \eqref{resolventP}  $u\in \buc$ which satisfies the non-expansive property. In other words, $\ran(I+\Delta t \hat{H})=\buc.$
\end{enumerate}
\end{cor}
\begin{proof}
First, (i) is directly from Proposition \ref{prop_h_limit}.

Second,  
taking limit $h\to 0$ preserves the nonexpansive property of $J_{\Delta t}$. Indeed, assume $u_1$ and $u_2$ are two solutions to the resolvent problem \eqref{resolventP} with $f_1(\vx)$, $f_2(\vx)$ respectively. Then denote $f_1^h(\vxi) := f_1(\vxi),\,f_2^h(\vxi) := f_2(\vxi)$, the associated discrete solutions $u^h_1(\vxi)$ and $u^h_2(\vxi)$, 
 by Lemma \ref{lem_nonexp_d}, satisfies 
 \begin{equation}
 \|u^h_1(\cdot)-u^h_2(\cdot)\|_{\ell^\8} \leq \|f_1^h(\cdot)-f_2^h(\cdot)\|_{\ell^\8}.
 \end{equation}
 Then taking limit $h\to 0$ implies
 \begin{equation}
 \|u_1-u_2\|_{L^\8} \leq \lim_{h\to 0} \|u^h_1-u^h_2\|_{\ell^\8} \leq \lim_{h\to 0} \|f_1^h - f_2^h\|_{\ell^\8} = \|f_1-f_2\|_{L^\8}.
 \end{equation}
 This implies the monotonicity as in Lemma \ref{lem_nonexp_d}.
 
 Third, notice $\ccc$ is dense in $\buc$. For any $f\in \buc$, there exists a Cauchy sequence $f_k\in \ccc$ such that
$\|f_k - f\|_{\8} \to 0$. 
Then from the last step, we know
 the corresponding viscosity solutions $u_k\in \ccc$ of the resolvent problem with $f_k$ is also a Cauchy sequence in $\buc$. Take $u=\lim_{k\to +\8}u_k$ as the limit of $u_k$ in $\buc$. The non-expansive property is still maintained as in the last step.
\end{proof}

\subsubsection{$\Delta t \to 0$ converges to the semigroup solution}
Now we follow the framework in \cite{crandall1983viscosity, feng2006large} to obtain a strongly continuous nonlinear semigroup solution $u(\vx,t)$ to HJE.
It is indeed a viscosity solution to the original dynamic HJE \cite{crandall1983viscosity, feng2006large}. Given any $t>0$, denote $[\frac{t}{\Delta t}]$ as the integer part of  $\frac{t}{\Delta t}$. We 
next prove the convergence from the discrete solution of the backward Euler scheme \eqref{bEuler} to the viscosity solution of HJE \eqref{HJE2psi}.
\begin{thm}\label{thm_vis}
Assume ${\Phi}^\pm_j(\vx), \, \vx\in \bR^N$, after zero extension, is local Lipschitz continuous. Assume there exists (componentwisely) positive $\vm$ satisfying \eqref{mb_j}.
Given any $t>0$ and $\Delta t>0$, let $\psih^n(\vxi), \vxi\in\Omega_{\V}, n=1,\cdots,[\frac{t}{\Delta t}]$ be a solution to \eqref{bEuler} with initial data $\psih^0(\vxi)\in \ellc$. Assume  for any $\vx\in \mathbb{R}^N$, $u^0(\vx) = \lim_{\vxi\to \vx} \psih^0(\vxi)$ and  $u^0\in \ccc$.  Let $\hat{H}: D(\hat{H})\subset \buc \to \buc$ be the viscosity extension of $H$ defined in \eqref{operatorH}. Then  as $h\to 0$, $\Delta t \to 0$, we know
\begin{equation}
u(\vx,t) = \lim_{\Delta t\to 0} \bbs{(I - \Delta t \hat{H})^{-[t/\Delta t]} u^0 } = \lim_{\Delta t\to 0}\bbs{\lim_{h\to 0} (I - \Delta t H_{\V})^{-[t/\Delta t]} \psih^0 }\in \buc.
\end{equation}
This is the unique viscosity solution to HJE \eqref{HJE2psi}.
\end{thm}

\begin{proof}
First, take $u^0\in \ccc$, fix $\Delta t$ and denote the solution to resolvent problem \eqref{resolventP_n} as $\psih^1(\vxi)$ satisfying
\begin{equation}
(I-\Delta t H_{\V}) \psih^1(\vxi) = \psih^0(\vxi)\in \ellc.
\end{equation}
Then by Proposition \ref{prop_h_limit}, since $u^0(\vx) = \lim_{\vxi\to \vx} \psih^0(\vxi)\in \ccc$, we know as $h\to 0$, $u^1\in \ccc$ is a viscosity solution to the continuous resolvent problem
\begin{equation}
(I-\Delta t H) u^1(\vxi) = u^0(\vxi).
\end{equation}
In the resolvent form, we have
\begin{equation}\label{maximal}
u^1 = (I-\Delta t \hat{H})^{-1}u^0 = \lim_{h\to 0} (I-\Delta H_{\V})^{-1}\psih^0\in \ccc .
\end{equation}
Repeating $[\frac{t}{\Delta t}]=:n$ times, we conclude that for any $\vx$ and $\vxi\to \vx$ as $h\to 0$,
\begin{equation}\label{conver_be}
u^n(\vx)=\bbs{(I - \Delta t \hat{H})^{-[t/\Delta t]} u^0}(\vx) = \lim_{h\to 0} \bbs{(I - \Delta t H_{\V})^{-[t/\Delta t]} \psih^0}(\vxi) = \lim_{h\to 0} u^n_{\V}(\vxi)\in \ccc.
\end{equation}
Second, the non-expansive property and the maximality of $-\hat{H}$ in $\buc$ are proved in Corollary \ref{cor:semi}.
Thus we conclude $-\hat{H}$ is m-accretive operator on  $\buc$. From inclusion \eqref{inc2} in Lemma \ref{lem_inc}, we know $u^0 \in \buc\subset \overline{D(\hat{H})}^{\|\cdot\|_\8}$, then by Crandall-Liggett's nonlinear semigroup theory
 \cite{CrandallLiggett}, $\hat{H}$  generates a strongly continuous nonexpansive semigroup 
\begin{equation}
u(\vx,t) = \lim_{\Delta t\to 0} (I - \Delta t \hat{H})^{-[t/\Delta t]} u^0 =:S(t)u^0,
\end{equation}
which is the solution to HJE \eqref{HJE2psi}. 
\end{proof}

\section{Short time classical solution and error estimate}\label{sec5}

In this section, we give   error estimates between the discrete solution of monotone scheme \eqref{upwind} (i.e., discrete nonlinear semigroup) and the classical solution of HJE for a short time.

Notice the chemical flux has a polynomial growth at the far field so usual methods for error estimate do not work. However, we observe that for special initial data $u^0\in C^1_{c}(\bR^{N*})$, the dynamic classical solution $u$ still belongs to $C^1_{c}(\bR^{N*})$ for a short time.
This can be seen from the characteristic method for a short time classical solution. Starting from any initial data $(\vx(0), \vp(0))$, construct the bi-characteristics $\vx(t), \vp(t)$ for HJE
\begin{equation}\label{cc}
\begin{aligned}
\dot{\vec{x}} = \nabla_p H(\vec{p}, \vec{x}), \quad \vec{x}(0) = \vec{x}_0,\\
\dot{\vec{p}} = -\nabla_x H(\vec{p}, \vec{x}), \quad \vec{p}(0) = \nabla \psi_0(\vec{x}_0)
\end{aligned}
\end{equation}
upto some $T>0$ such that the characteristics exist uniquely.   Then along characteristics,  with $\vec{p}(t) = \nabla_xu( \vec{x}(t), t )$,  we know $z(t)=u(\vx(t),t)$ satisfies
\begin{equation}\label{tm_zz}
\begin{aligned}
\dot{z} =\nabla_x u(\vx(t),t) \cdot \dot{\vx} + \pt_t u(\vx(t),t) =\vec{p} \cdot \nabla_p H(\vec{p},\vec{x}) +H(\vec{p}, \vec{x}), \quad z(0) = u_0(\vec{x}_0).
\end{aligned}
\end{equation} 
Hence after solving \eqref{cc}, we can solve for $z(t)=u(\vx(t),t).$
 
Since $u^0\in C^1_{c}(\bR^{N*})$, we know for $|\vx_0|\gg 1$, $\vp_0 = 0$. Then as long as bi-characteristic is unique up to $T$, $\vp \equiv 0$ due to $\nabla_x H(\vec{0}, \vx)=0$. This, together with \eqref{tm_zz} and $H(\vec{0}, \vx)=0$, implies $\dot{z} = 0 $ up to $T$. Thus we conclude 
\begin{equation}\label{compact}
u(\vx,t)= \text{ const  \,\, for } |\vx|\gg 1.
\end{equation}

Recall $\Delta t$ is the time step and the size of a container for chemical reactions is  $\frac{1}{h}$.
Next, based on the above observation, we give the error estimate for initial data $u^0\in C^1_{c}(\bR^{N*}).$
\begin{prop}
Assume ${\Phi}^\pm_j(\vx), \, \vx\in \bR^N$, after zero extension, is locally Lipschitz continuous.
Let $T>0$ be the maximal time such that bi-characteristic $\vx(t), \vp(t)$ for HJE exists uniquely.    Given any initial data $\psih^0$ for the backward equation  \eqref{upwind}, assume there exists $u^0$ such that $\|\psih^0 -u^0\|_{\8} \leq Ch$. Then  we have
\begin{enumerate}[(i)]
\item the error estimate between backward Euler scheme solution $\psih^n(\vxi)$ to \eqref{bEuler} and the classical solution $u(\vx, t^n)$ to HJE \eqref{HJE2psi}
\begin{equation}
\begin{aligned}
\|\psih^n - u(\cdot, t^n)\|_{\8} 
\leq C (T+1) (\Delta t +h). 
\end{aligned}
\end{equation}
\item the error estimate between monotone scheme solution $\psih(\vxi, t)$ to  \eqref{upwind} and the classical solution $u(\vx, t)$ to HJE \eqref{HJE2psi}
\begin{equation}
\|\psih(\cdot,t)  - u(\cdot, t)\|_{\8} \leq C h, \quad \forall 0<t\leq T.
\end{equation}
\end{enumerate}
\end{prop}

\begin{proof}
Let $T$ be the maximal time  such that bi-characteristic $\vx(t), \vp(t)$ for HJE exists uniquely and thus from the uniqueness of solution to HJE, the viscosity solution obtained in Theorem \ref{thm_vis} is the unique classical solution. Then from analysis for \eqref{compact}, $u$ is constant outside $B_R$.

(i) We perform the truncation error estimate for $u(\vx, t)\in C^2$. By Taylor expansion, plugging  $u(\vx, t^n)$ to the resolvent problem, we have
\begin{equation}
u(\vx, t^n) - \Delta t H(\vx, \nabla u(\vx, t^n)) = u(\vx,t^{n-1}) + O(\Delta t^2).
\end{equation} 
Then we further plug $u(\vxi, t^n)$ into the discrete resolvent problem. Since $H=0$ for $\vx\in B_R^c$, the polynomial growth outside $B_R$ does not affect the truncation error and thus
 \begin{equation}
u(\vxi, t^n) - \Delta t H_{\V}(\vxi, \nabla u(\vxi, t^n)) = u(\vxi,t^{n-1}) + O(\Delta t^2+  h\Delta t ).
\end{equation}
(ii) We perform error estimate using the nonexpansive property for the discrete resolvent in Lemma \ref{lem_nonexp_d}. Recall the numerical solution obtained in backward Euler scheme
\begin{equation}\label{bEuler_n}
 \psih^{n}(\vxi) -\Delta t  H_{\V}(\vxi, \psih^{n}(\vxi), \psih^n) = \psih^{n-1}(\vxi). 
\end{equation}
Then by Lemma \ref{lem_nonexp_d}, we have
\begin{equation}
\begin{aligned}
\|\psih^n - u(\cdot, t^n)\|_{\8} &\leq \|\psih^{n-1} - u(\cdot, t^{n-1})\|_{\8} + C\Delta t(\Delta t + h) \\
&\leq \cdots \leq \|\psih^{0} - u(\cdot, 0)\|_{\8} + C n\Delta t(\Delta t +h)
\leq C (T+1) (\Delta t +h). 
\end{aligned}
\end{equation}
Notice this is linear growth in time. Changing the above resolvent problem to the  original CME in the HJE form \eqref{upwind}, we have the first order convergence
\begin{equation}
\|\psih(\cdot, t)  - u(\cdot, t)\|_{\8} \leq C h, \quad \forall 0<t\leq T.
\end{equation}
\end{proof}

Under the observation of the finite time propagation of support \eqref{compact},  the error estimates for monotone schemes to classical solutions of HJE are standard; see abstract theorem in \cite{souganidis1985approximation}.   We remark that if the solution to  HJE \eqref{HJE2psi} lose the regularity and becomes only locally Lipschitz, then the convergence rate in terms of $h$ can be at most $h^{\frac12}$; c.f. \cite{crandall1984two, waagan2008convergence, souganidis1985approximation}.

\section{Application to large deviation principle for chemical reactions at single time}
\label{sec6}
In previous sections, we proved the convergence from the solution to  monotone scheme \eqref{upwind}, i.e., Varadhan's nonlinear semigroup \eqref{semigroup}, to the global viscosity solution to HJE \eqref{HJE2psi}. In this section, we discuss some straightforward applications of this convergence result, particularly the large deviation principle and the law of large numbers at a single time. 

First, we will discuss the exponential tightness of the large size process $\cv$ at single times.  
Then, using the   Lax–Oleinik semigroup representation for first order HJE, the convergence result in Theorem \ref{thm_vis} implies  the convergence from Varadhan's discrete nonlinear semigroup to the continuous Lax–Oleinik semigroup. Therefore, together with the exponential tightness, we obtain the large deviation principle at a single time point in Theorem \ref{thm_dp}.  As long as one can prove the exponential tightness in Skorokhod space $D([0,T];\Omega_{\V})$,  by using \cite[Theorem 4.28]{feng2006large},   the large deviation principle for finite many time points will lead to the  sample path large deviation principle for the chemical reaction process $\cv.$ We leave the  exponential tightness in Skorokhod  space $D([0,T];\Omega_{\V})$ to future study.    The  exponential tightness for $(\cv(\cdot))$ in path space was  proved in \cite[Lemma 2.1]{Dembo18} under  the assumption that there exists a growth estimate for a Lyapunov function.  For finite states continuous time Markov chain, the exponential tightness for $(\cv(\cdot))$ in path space was  proved by \cite{Mielke_Renger_Peletier_2014}.
  
Second, the large deviation principle also implies the mean-field limit of the CME \eqref{rp_eq} is the RRE \eqref{odex}, but with a more explicit rate of the concentration of measures. The mean-field limit of the CME for chemical reactions was first proved by Kurtz \cite{kurtz1970solutions, Kurtz71}. Recent results in \cite{{maas2020modeling}}   proves the evolutionary $\Gamma$-convergence from   CME to the Liouville equation corresponding to RRE under the detailed balance assumption.

\subsection{Large deviation principle for chemical reaction at single times}
 In this section, we prove the large deviation principle of the random variable $\cv(t)$ at each time. Precisely, 
\begin{defn}\label{def_ld} 
Let $\cv $ be the large size process   defined in \eqref{Csde} and  $\cv (0)=\vx_0^{\V}\in\bR^N_+$. Assume $\vx_0^{\V} \to \vx_0 \in \bR^N_+$.  Then at each time $t$, we say the random variable $\cv (t)$  satisfies the large deviation principle in $\bR^N_+$ with a good rate function $I(\vy; \,\vx_0,t)$ defined in \eqref{LO} if
 for any open set $\mathcal{O}\subset  \bR^N_+$, it holds
\begin{equation}
\liminf_{h\to 0} h \log \bP_{\vx_0^{\V}} \{\cv (t) \in  \mathcal{O}\} \geq - \inf_{\vy\in \mathcal{O}} I(\vy;\,\vx_0, t)
\end{equation}
and for any closed set $\mathcal{C}\subset \bR^N_+$, it holds
\begin{align}
\limsup_{h \to 0} h \log \bP_{\vx_0^{\V}} \{\cv (t) \in\mathcal{C}\} \leq - \inf_{\vy\in \mathcal{C}} I(\vy;\,\vx_0, t).
\end{align}
\label{LD11}
Here  $I(\vy; \vx_0,t)$ is a good rate function means the sublevel set $\{\vy\in \bR^N_+; \, I(\vy;\vx_0, t)\leq \ell\}$ is compact for any $\ell$. This sublevel set compactness automatically implies $I$ is lower semicontinuous.
\end{defn}

From the theory of Hamilton-Jacobi equation (or the deterministic optimal control formulation), it is well-known that the viscosity solution to first order HJE can be expressed as the Lax-Oleinik semigroup, cf.   \cite[Theorem 2.22]{Tran21},\cite{evans2008weak}.
\begin{lem}[Lax-Oleinik semigroup]
Assume ${\Phi}^\pm_j(\vx), \, \vx\in \bR^N$, after zero extension, is locally Lipschitz continuous.
The viscosity solution $u(\vx,t)$ to HJE \eqref{HJE2psi} can be represented as the   Lax–Oleinik semigroup, i.e., \eqref{LO}
\begin{equation}
u(\vx,t) = \sup_{\vy  } \bbs{u_0(\vy) - I(\vy; \vx,t)},  \quad I(\vy; \vx,t) := \inf_{\gamma(0)=\vx, \gamma(t)=\vy} \int_0^t L(\dot{\gamma}(s),\gamma(s)) \ud s,
\end{equation}
where $L(\vs,\vx)$ is the convex conjugate of $H(\vp,\vx)$.
\end{lem}
We can directly verify the following semigroup property for the Lax-Oleinik representation for viscosity solution \eqref{LO}
\begin{equation}
u(\vx,t) = \sup_{\vy} \bbs{ u(\vy, t-\tau) - I(\vy; \, \vx, \tau) }, \quad 0\leq \tau \leq t.
\end{equation}
This is also known as the dynamic programming principle.

\begin{thm}\label{thm_dp}
Assume ${\Phi}^\pm_j(\vx), \, \vx\in \bR^N$, after zero extension, is locally Lipschitz continuous. Assume there exists (componentwisely) positive $\vm$ satisfying \eqref{mb_j}.
 Let $\cv (0)=\vx_0^{\V} \to \vx_0$ in $\bR^N_+$. Then the chemical reaction process $\cv(t)$ at each time $t$ satisfies the  large deviation principle  with a good rate function $I(y;\vx_0,t)$ as in Definition \ref{def_ld}.
\end{thm}

\begin{proof}
First,   the    exponential tightness of the process $\cv$ is essential for large deviation analysis.
Using Lemma \ref{tight2}, the existence of positive $\vm$ and the fixed initial position $\cv (0)$ ensure that   $\cv$ is   exponentially tight for any $t$.

Second, 
from Varadhan's lemma \cite{Varadhan_1966} on the necessary condition of the large deviation principle and Bryc's theorem \cite{Bryc_1990} on the sufficient condition, we know $\cv$ satisfies the large deviation principle with a good rate function $I(y)$, if and only if for any bounded continuous function $u_0(\vx)$, 
\begin{equation}\label{var}
\lim_{h\to 0} h \log \bE^{\vx}\bbs{ e^{\frac{u_0(\cv(t))}{h}} } =u(\vx,t) = \sup_{y} \bbs{u_0(y) - I(y; x,t)}.
\end{equation}

Then we show \eqref{var} holds. From the WKB reformulation for the backward equation \eqref{wkb_b}
 $$\uh(\vxv,t) =h \log w_{\V}(\vxv,t) = h \log \bE^{\vxv}\bbs{e^{ \frac{u_0(\cv_t)}{h}}},$$
we show below the pointwise convergence  
 for any $t\in [0,T]$,
\begin{equation}\label{con1}
 \lim_{h\to 0} h \log \bE^{\vx}\bbs{ e^{\frac{u_0(\cv_t)}{h}} } =u(\vx,t).
\end{equation} 
 
For any $t\in[0,T]$,  $\eps>0$, from Theorem \ref{thm_vis}, we know there exists $\Delta t_0$ such that for any $\Delta t\leq \Delta t_0$, $n=[t/\Delta t]$, the estimate between the solution $u(\vx,t)$ to HJE and the solution $u^n(\vx)$ to the backward Euler approximation \eqref{resolventP} satisfies
$
\|u(\cdot,t) - u^n(\cdot)\|_{L^\8} \leq \frac13\eps.
$
For this $u^n(\vx)$, at any fixed $\vx\in \mathbb{R}^N$, from \eqref{conver_be}, there exists $h_0$ such that for any $h\leq h_0$, the solution $u_{\V}^n(\vxi)$ to the discrete \eqref{bEuler} converges to $u^n(\vx)$ pointwisely. That is to say, exists $h_0$, such that for $h\leq h_0$,
$
|u^n(\vx)-u^n_{\V}(\vxi)| \leq \frac{1}{3}\eps.
$
From the convergence of backward Euler scheme \eqref{bEuler} to the ODE system \eqref{odex} in \eqref{semi27}, we obtain another $\frac13 \eps$ and hence \eqref{con1} holds.
 
 Meanwhile, $u(x,t)$ has the   Lax–Oleinik semigroup representation \ref{LO}. Hence we conclude \eqref{var}.
\end{proof}

As a consequence of the above large deviation principle, we state below the law of large numbers gives the mean-field limit ODE \eqref{odex}.
\begin{cor}\label{cor:ode}
Assume ${\Phi}^\pm_j(\vx), \, \vx\in \bR^N$, after zero extension, is locally Lipschitz continuous. Assume there exists (componentwisely) positive $\vm$ satisfying \eqref{mb_j}.  Let $\cv (0)=\vx_0^{\V} \to \vx_0$ in $\bR^N_+$.
Let $\vx^*(t)$ be the solution to RRE \eqref{odex}. Then $\vx^*(t)$ is the mean-field limit path in the sense that the weak law of large numbers for process $\cv(t)$ holds at any time $t$. {\blue Moreover, for any $t>0$, for any $\eps>0$, there exists $h_0>0$ such that if $h\leq h_0$ then
$$ \bP \{ |\cv(t) - \vx^*(t) | \geq \eps \} \leq e^{ -\frac{\alpha(\eps,t)}{2h} }, $$
where $\alpha(\eps,t) = \inf_{|\vy - \vx^*(t)|\geq \eps} I(\vy; \vx,t) >0.$
}
\end{cor}
The proof for  mean-field limit of $\cv(t)$ was given by Kurtz \cite{Kurtz71}. Here we use the large deviation principle at singe time $t$ to illustrate that large deviation principle implies the law of large numbers after the concentration of measure and gives the exponential rate of that concentration.
This can be seen  via verification method. 

First, we know  $\vx^*(t)$ is the solution to RRE \eqref{odex} if and only if the action cost is zero $I(y; x_0, t)=0$; see Appendix \ref{app2}.

Second, we choose any bounded continuous function $u_0$ in \eqref{var} such that at time $t$, $u_0(\vx^*(t))=\argmax_x u_0(x).$ Then \eqref{var} becomes
$\lim_{h\to 0} h\log \bE(e^{u_0(\cv(t))/h})=u_0(\vx^*(t))$. This is to say $\cv(t)$ concentrates at $\vx^*(t)$. Thus the concentration rate near a tube of $\vx^*$ is directly given by the obtained rate function $I(\vy;\vx,t)$.
 
We  remark that the sample path large deviation principle for $\{\cv(t)\}$ is more significant for studying the transition path theory and the proof is more involved; see \textsc{Agazzi} et.al. \cite{Dembo18}. We state the definition of the sample path large deviation principle in Appendix \ref{app_ld} and show it covers the single time large deviation principle in Theorem \ref{thm_dp}.


\section*{Acknowledgment}
The authors would like to thank Jin Feng for valuable discussions. 
The authors would also like to thank the Isaac Newton Institute for Mathematical Sciences at Cambridge for their support and hospitality during the program "Frontiers in kinetic theory: connecting microscopic to macroscopic scales - KineCon 2022," where work on this paper was partially undertaken. Yuan Gao was supported by NSF under Award DMS-2204288. Jian-Guo Liu was supported by NSF under Award DMS-2106988.

\bibliographystyle{alpha}
\bibliography{LD_vis_bib}

\appendix 

\section{Terminologies for macroscopic RRE, LDP and viscosity solutions}
  In this appendix, we review some known terminologies  for the large size limiting RRE \eqref{odex} and collect some preliminary results/concepts for the   detailed  balanced RRE system, and the associated Hamiltonian. We also give the definition for viscosity solutions and LDP for completeness.

\subsection{The macroscopic RRE and equilibrium}\label{app1}
 The $M\times N$ matrix $\nu:= \bbs{\nu_{ji}}=(\nu_{ji}^- - \nu_{ji}^+), j=1,\cdots,M, \, i=1,\cdots, N$ is called the Wegscheider matrix and $\nu^T$ is referred as the stoichiometric matrix \cite{mcquarrie1997physical}.  Then mass balance \eqref{mb_j} reads
 $\nu \vec{m}=\vec{0}$ and
  the  Wegscheider matrix $\nu$   has a nonzero kernel, i.e., 
$
\dim \bbs{\kk(\nu)} \geq 1.
$
Thus we have a direct decomposition for the species space
\begin{equation}\label{direct}
\bR^N = \ran(\nu^T)\oplus \kk(\nu),
\end{equation}
where $\ran(\nu^T)$ is the span of the column vectors $\{\vec{\nu}_j\}$ of $\nu^T$.
Denote the stoichiometric space $G:=\ran(\nu^T)$. Given an initial state $\vec{X}_0 \in \vq+G$, $\vq\in\kk(\nu)$, the dynamics of both mesoscopic \eqref{Csde} and macroscopic \eqref{odex} states stay in the same space $G_q:=\vq+G$, called a stoichiometric compatibility class.

\subsubsection{Detailed balance and characterization of steady states  for the macroscopic RRE}
Denote a steady state to RRE \eqref{odex} as $\peq$, which satisfies
\begin{equation} 
 \sum_{j=1}^M \vec{\nu}_j \bbs{\Phi^+_j(\peq) - \Phi^-_j(\peq)} =0.
\end{equation}
Under the assumption of detailed balance \eqref{DB} and \eqref{lma}, we have
\begin{equation}
\log k_j^+ -  \log k_j^- =   \vec{\nu}_{j} \cdot \log \peq
\end{equation}
Then $\vec{\nu}_{j} \cdot \log (\peq_1-\peq_2)=0$ and all the steady states of RRE \eqref{odex} can be characterized as follows.
For any $\vq\in \kk(\nu)$, there exists a unique steady states $\peq_*$ in the space $ \{\vx \in \vq+G;\,  \vx>0\}$.
This uniqueness of steady states in one stoichiometric compatibility class is a
  well-known result; c.f., \cite[Theorem 6A]{Horn72}, \cite[Theorem 3.5]{Kurtz15}.

\subsection{Properties of Hamiltonian $H$ and the convex conjugate}\label{app2}
Recall  the mass conservation law of chemical reactions  \eqref{mb_j} and   direct decomposition \eqref{direct}.  We further observe, for $H$, we have
$
H(\vec{0},\vec{x})\equiv 0$ and hence $\nabla_x H(\vec{0}, \vec{x}) \equiv 0.
$
  We summarize useful lemmas on the properties of $H$ and $L$.
\begin{lem}\label{lem_Hdege}
Hamiltonian $H(\vp,\vx)$ on $\bR^N\times \bR^N$ in \eqref{H} is degenerate in the sense that
\begin{equation}\label{Hdege}
H(\vp, \vx) = H(\vp_1, \vx),
\end{equation}
where $\vp_1\in \ran(\nu^T)$ is the direct decomposition of $\vp$ such that
\begin{equation}\label{decom}
\vp = \vp_1 + \vp_2, \quad \vp_1\in \ran(\nu^T),\,\, \vp_2 \in \kk(\nu).
\end{equation} 
\end{lem}

\begin{lem}\label{Hconvex}
For any $\vx>0$, $H(\vp,\vx)$ defined in \eqref{H} is strictly convex and exponential coercivity w.r.t.  $\vp \in G$, i.e., for $\vp\in G$, there exists $A>0$ such that ${H(\vp,\vx)} \geq  A e^{\alpha |\vp|}$ for $|\vp|\gg 1$.
\end{lem}

Indeed, fix any $\vx>0$,
$
   {H(\vp,\vx)}  =   \sum_{j=1}^M \Phi_j^+(\vx)  \bbs{e^{\vec{\nu}_j \cdot \vp}-1}  + \Phi^-_j(\vx) \bbs{e^{-\vec{\nu}_j \cdot \vp}-1}. 
$
Since for $\vp\in G$, $\vec{\nu}_j \cdot \vp \neq 0$ and thus $|\vec{\nu}_j \cdot \vp | > \alpha |\vp| >0$ for some $j$. Hence for $|\vp|$ sufficient large, 
\begin{equation}
  {H(\vp,\vx)} \geq   \frac12   \min\{\Phi_j^+(\vx), \Phi^-_j(\vx) \}   e^{\alpha |\vp|}.   
\end{equation}
 
Since for $\vx>0$, $H$ defined in \eqref{H} is  convex and superlinear w.r.t $\vp$,  we compute the convex conjugate of $H$ via the Legendre transform. For any $\vs\in \bR^N$, define
\begin{equation}\label{L}
L(\vs,\vec{x}) := \sup_{\vec{p}\in \bR^N} \bbs{ \la \vp ,  \vs \ra - H(\vp, \vx)} 
\end{equation}

Then we have the following lemma.
\begin{lem}\label{lem:least}
For $L$ function defined in \eqref{L}, we know
 $L(\vs,\vx)\geq 0$ and 
\begin{equation}\label{LL}
L(\vs, \vx) = \left\{ \begin{array}{cc}
\max_{\vp \in G} \{ \vs \cdot \vp - H(\vp, \vx) \}<+\8, & \vs\in G \text{ and } \vx>0,\\
- \min_{\vp\in \bR^N} H(\vp,\vx)<+\8, & \vs = \vec{0},\\
+\8, & \text{ otherwise };
\end{array}  \right.
\end{equation}
moreover, $L$ is strictly convex in $G$ for $\vx>0$.
\end{lem}

\begin{proof}
 First, for $\vs= \vec{0}$, $L(\vec{0},\vx)= - \min_{\vp\in \bR^N} H(\vp,\vx)$ due to the coercivity of $H$. For $\vs\neq \vec{0}$, case (i), for some component $x_i\leq 0$, by the zero extension of $H$, we have $L(\vs,\vx)=+\8$; case (ii), if $\vs \in \ker(\nu)$, then take sup for $\vp\in \ker(\nu)$ gives $L(\vs,\vx)=+\8$.
For $\vx>0, \vp\in G$, since $H$ is superlinear, the maximum is attained at $\vp^*(\vs, \vx)$ satisfying
\begin{equation}\label{ts1}
\vs = \nabla_p H(\vp^*, \vx) = \sum_j \vec{\nu}_j \bbs{\Phi_j^+ e^{\vec{\nu}_j \cdot \vp^*} - \Phi_j^- e^{-\vec{\nu}_j \cdot \vp^*} }.
\end{equation}
Thus for $\vx>0,$
$
L(\vs, \vx) = \vs \cdot \vp^*(\vs,\vx) - H(\vp^*(\vs,\vx),\vx).
$
\end{proof}

Assume there is a $C^1$ least action curve $\gamma(t)$ connecting  $\gamma(0)=\vx> 0$ and $\gamma(t)=\vy> 0$ and $\gamma(\tau) > 0$ for all $\tau\in[0,t]$. Thus $0\leq I(\vy;\vx,t)<+\8$.  Then $\gamma(\tau)$ satisfies Euler-Lagrangian 
equation
\begin{equation}\label{EL}
\frac{\ud}{\ud t} \bbs{\frac{\pt L}{\pt \dot{\vx}}(\dot{\vx}(t), \vx(t))  }= \frac{\pt L}{\pt \vx}(\dot{\vx}(t), \vx(t)).
\end{equation}

\subsection{Concepts of viscosity solution}\label{app3}
Here for the convenience of readers, we adapt the notion of viscosity solutions used in \cite{Barles_Perthame_1987}.
\begin{defn}
Let $\Omega\subset \mathbb{R}^N$ be an open set. We say $u\in \usc(\Omega)$ (resp. $\lsc(\Omega)$) is a viscosity subsolution (resp. supersolution) to \eqref{resolventP}, if for any $x\in \Omega$ and any test function $\varphi\in C^\8(\mathbb{R}^N)$ such that $u-\varphi$ has a local maximum (resp. minimum) at $x$, it holds
\begin{equation}
u(x) - H(\nabla\varphi(x), x) \leq 0. \,\, (\text{resp. }\geq 0)
\end{equation}
We say $u$ is a viscosity solution if it is both a viscosity subsolution and supersolution. 
\end{defn}

\subsection{Definition of the sample path large deviation principle}\label{app_ld}
Let $D([0,T];\bR^N_+)$ be Skorokhod space and $AC([0,T];\bR^N)$ be the space of absolute continuous curves. We state the definition of the sample path large deviation principle in path space $D([0,T];\bR^N_+).$
\begin{defn}
Let $\cv $ be the large number process   defined in \eqref{Csde} with generator $Q_{\V}$. Assume $\cv (0)=\vx_0^{\V}\in \bR^N_+$.  Then we say the sample path $\cv (t)$, $t\in [0,T]$ satisfies the large deviation principle in $D([0,T];\bR^N_+)$ with the good rate function
\begin{equation}
A_{x_0, T}(\vx(\cdot)) := \left\{ 
\begin{array}{cc}
\int_0^T L(\dot{\vx}(t),\vx(t)) \ud t & \text{ if } \vx(0)=\vx_0,\,\, \vx(\cdot)\in AC([0,T];\bR^N),\\
+\8 & \text{ otherwise}
\end{array}
\right. 
\end{equation} 
if for any open set $\mathcal{E}\subset D([0,T];\bR^N_+)$, it holds
\begin{equation}
\liminf_{h\to 0} h \log \bP_{\vx_0^{\V}} \{\cv (t) \in  \mathcal{E}\} \geq - \inf_{\vx\in \mathcal{E}} A_{\vx_0, T}(\vx(\cdot)),
\end{equation}
while for any closed set $\mathcal{G}\subset D([0,T];\bR^N_+)$, it holds
\begin{align}
\limsup_{h\to 0} h \log \bP_{\vx_0^{\V}} \{\cv (t) \in \mathcal{G}\} \leq - \inf_{\vx\in \mathcal{G}} A_{\vx_0, T}(\vx(\cdot)).
\end{align}
\label{LD22}
\end{defn}
Under additional mild assumption for exponential tightness in $D([0,T];\bR^N_+)$, the sample path large deviation principle for $\cv(y)$ was proved in \cite[Theorem 1.6]{Dembo18}. It covers the  single time large deviation principle  in Theorem \ref{thm_dp}. Indeed, for any fixed open set $\mathcal{O}\subset \bR^N_+$, we take a special open set $\mathcal{E}\subset D([0,T];\bR^N_+) $ as $\mathcal{E} = \{\vx(\cdot)\in D([0,T];\bR^N_+); \, \vx(t) \in \mathcal{O}\}$.  Then
$$\inf_{\vx\in \mathcal{E}} A_{\vx_0, T}(\vx(\cdot)) = \inf_{\vy \in \mathcal{O}} \bbs{\inf_{\vx(\cdot)\in D([0,T];\bR^N_+),\, \vx(0)=\vx_0,\, \vx(t) = \vy } \int_0^T L(\dot{\vx}(s),\vx(s)) \ud s} = \inf_{\vy \in \mathcal{O}} I(\vy;\,\vx_0, t). $$
Here in the last equality, the least action from $0$ to $T$ is the combination of  the least action from $0$ to $t$ and a zero-cost action for $t$ to $T$.

\end{document}